\documentclass[aap]{imsart}

\RequirePackage{amsthm,amsmath,amsfonts,amssymb}
\RequirePackage[numbers]{natbib}

\RequirePackage[colorlinks,citecolor=blue,urlcolor=blue]{hyperref}
\RequirePackage{graphicx}
\usepackage{amsfonts}
\usepackage{hyperref}
\usepackage{amsmath}
\usepackage{xcolor}
\usepackage{amsthm}
\usepackage{pdflscape}
\usepackage{pgfplots}
\usepackage{mathrsfs}
\usepackage{lettrine}
\usepackage{scrextend}
\usepackage{enumitem}
\usepackage{bbm}
\usepackage{mathtools}

\usepackage{graphicx}
\usepackage{subcaption}
\usepackage{mwe}

\startlocaldefs
\theoremstyle{plain}
\newtheorem{axiom}{Axiom}
\newtheorem{remark}{Remark}
\newtheorem{claim}[axiom]{Claim}
\newtheorem{theorem}{Theorem}[section]
\newtheorem{lemma}[theorem]{Lemma}
\newtheorem{proposition}[theorem]{Proposition}
\newtheorem{corollary}[theorem]{Corollary}
\theoremstyle{remark}
\newtheorem{definition}[theorem]{Definition}
\newtheorem*{example}{Example}
\newtheorem*{fact}{Fact}

\endlocaldefs

\newcommand\myeq{\stackrel{\mathclap{\normalfont\mbox{\itshape{d}}}}{=}}

\begin{document}

\begin{frontmatter}
\title{Coupling from the past\\ for the null recurrent Markov chain}
\runtitle{CFP for the null recurrent Markov chain}

\begin{aug}
\author[A]{\fnms{Fran\c{c}ois} \snm{Baccelli}\ead[label=e1]{francois.baccelli@ens.fr}},
\author[B]{\fnms{Mir-Omid} \snm{Haji-Mirsadeghi}\ead[label=e2]{mirsadeghi@sharif.ir}}
\and
\author[C]{\fnms{Sayeh} \snm{Khaniha}\ead[label=e3]{sayeh.khaniha@ens.fr}}
\address[A]{INRIA/ENS Paris,
\printead{e1}}

\address[B]{Sharif University of Technology,
\printead{e2}}

\address[C]{INRIA/ENS Paris,
\printead{e3}}

\end{aug}

\begin{abstract}
The Doeblin Graph of a countable state space Markov chain describes the joint pathwise evolutions of the Markov dynamics starting from all possible initial conditions, with two paths coalescing when they reach the same point of the state space at the same time. Its Bridge Doeblin subgraph only contains the paths starting from a tagged point of the state space at all possible times. In the irreducible, aperiodic, and positive recurrent case, the following results are known: the Bridge Doeblin Graph is an infinite tree that is unimodularizable. Moreover, it contains a single bi-infinite path which allows one to build a perfect sample of the stationary state of the Markov chain. The present paper is focused on the null recurrent case. It is shown that when assuming irreducibility and aperiodicity again, the Bridge Doeblin Graph is either an infinite tree or a forest made of a countable collection of infinite trees. In the first case, the infinite tree in question has a single end, is not unimodularizable in general, but is always locally unimodular. These key properties are used to study the stationary regime of several measure-valued random dynamics on this Bridge Doeblin Tree. The most important ones are the taboo random dynamics, which admits as steady state a random measure with mean measure equal to the invariant measure of the Markov chain, and the potential random dynamics which is a random extension of the classical potential measure, with a mean measure equal to infinity at every point of the state space. The practical interest of these two random measures is discussed in the context of perfect sampling.
\end{abstract}

\end{frontmatter}

\\\\

\section{Introduction}
Let $\{{X}_{t}\}_{t \in \mathbb{N}}$ be a Markov Chain with countable state space, $\mathcal{S}$. It is well known that when $\{{X}_{t}\}_{t \in \mathbb{N}}$ is  irreducible, aperiodic, and positive recurrent, it has a unique stationary distribution. On the other hand, when it is null recurrent, it admits no stationary probability distribution, but a unique stationary measure $\sigma$, i.e., the measure $\sigma$ satisfies $\sigma=\sigma P$ with $\sigma(s^{*})=1$, where $P$ is the transition probability matrix the Markov Chain, and $s^{*}$ is an arbitrary fixed point in $\mathcal{S}$.\\
One can consider a Markov Chain as a  dynamics on $\mathcal{S}$-valued random variables. This dynamics can be written as the following equation \footnote{Another representation for Equation \eqref{002} is the stochastic recurrence equation ${X}_{t+1}=h({X}_{t}, \xi_{t})$.}
\begin{equation}
\label{002}
X_{t+1}=\sum _{x \in \mathcal{S}} \mathbbm{1}_{ \{X_{t}=x\}} h(x,\xi_{t}^{x}), 
\end{equation}
where $\{\xi_{t}^{x}, x\in \mathcal{S}\}_{t\in \mathbb{Z}}$ is the source of randomness which can be assumed  i.i.d. for different $t \in \mathbb{Z}$ and such that $P[h(x,\xi_{t}^{x}=y)]=p_{xy}$. Here $h$ is an update rule which allows one to construct the random variable at time $t+1$ from that at time $t$. When $\{{X}_{t}\}_{t \in \mathbb{Z}}$ is positive recurrent, Equation \eqref{002} has a stationary solution. This means there is a random variable $X$, with distribution $\sigma$, and such that $X\overset{d}{=}h(X, \xi_{t}^{X})$, with $\overset{d}{=}$ meaning equality in distribution. In the null recurrent case, this dynamic does not admit such a stationary solution. This paper introduces two other dynamics related to the Markov Chain $\{{X}_{t}\}_{t \in \mathbb{Z}}$, that have a stationary solution in the recurrent case, including the null recurrent one.These dynamics are on random measures on $\mathcal{S}$ and are of the form 
\begin{equation}
\label{003}
{M}_{t+1}=H({M}_{t}, \xi_{t}),
\end{equation}
where for each $t$, $M_{t}$ is a random measure on $\mathcal{S}$ and $\xi_{t}=\{\xi_{t}^{x}\}_{x\in \mathcal{S} }$ is the same as in Equation \eqref{002}. Two different update rules for $H$, referred to as the Taboo Dynamics and Potential Dynamics are introduced. See Section \ref{Main} for their definitions. These two dynamics are related to the Doeblin coupling of the Markov Chain $\{{X}_{t}\}_{t \in \mathbb{Z}}$. They leverage the Doeblin Graph and a subgraph of it, the Bridge Graph, of the Markov Chain. The Doeblin Graph of  $\{X_{t}\}$ is a random graph with vertices in $\mathbb{Z} \times \mathcal{S}$. In this graph, the $x$-axis, which is referred to as the \emph{time axis}, represents time, and the $y$-axis represents the state space. The edges of the Doeblin Graph are defined using the transition probabilities of the Markov Chain: there is an edge from each vertex $(t,x)$ to vertex $(t+1,h(x,\xi_{t}^{x}))$, with $\{\ \xi_{t}^{x}\}_{x,t}$ as defined above. The Bridge Graph with respect to $s^{*}$ is the union of all paths of the Doeblin Graph starting from $ \mathbb{Z} \times  \{s^{*}\} $, where $s^{*}$ is an arbitrary fixed point in $\mathcal{S}$. \\
It is shown in \cite{Baccelli2018} that the Bridge Graph of an aperiodic, positive recurrent Markov Chain is a.s. a tree, which is locally finite and unimodularizable. Using the unimodular property, it is shown that in this case, there is a unique bi-infinite path in the this graph. The distribution of the points in this bi-infinite path is the stationary distribution.\\ 
The first aim of the present work is the study of the properties of the Bridge Graph constructed from an aperiodic and {\em null} recurrent Markov Chain. In this case, one can show that the Bridge Graph
is not unimodularizable in general, that it can be both a tree or a forest, and that it contains no bi-infinite path when it is connected or when it satisfies some additional condition given in the paper. It is also shown that the Bridge Graph is locally unimodular in that it contains a unimodular and one ended random tree, the regeneration time tree. This allows one to show that the Bridge Graph is one ended as well, which is essential for the following analysis.\\
In the recurrent case, two measure-value dynamics are defined on the Bridge Graph, the Taboo and the Potential dynamics. First, it is shown that the Taboo dynamics has a stationary state on the space of random measures on $\mathcal{S}$, called the Taboo Point Process (TPP). A key point is that the mean measure of the TPP is equal to the invariant measure of the Markov Chain.\\
The Potential Dynamics is also studied. In the null recurrent case, this dynamics also has a stationary state on the space of random measures on $\mathcal{S}$. This random measure is called the Potential Point Process (Potential PP). This point process is locally finite, but its mean measure is infinite. \\
These two point processes can also be defined in the positive recurrent case as well, and their properties are also discussed in this case.
They are hence fundamental objects in that they can be defined for all recurrent discrete time discrete space Markov Chains, they are left invariant by the Markov dynamics, and they provide, as it will be shown, key informations
on the CFTP algorithm as well as complementary informations on the two most important deterministic measures of Markov Chain theory, namely the invariant measure and the potential measure.\\

After studying the existence of stationary regimes for these dynamics, their uniqueness is also discussed. For this purpose, these dynamics are considered as Markov Chains on the space of locally finite counting measures on $\mathcal{S}$, $\mathcal{N}(\mathcal{S})$. It is shown that the dynamical systems introduced above may have other stationary solutions than the one built on the Bridge Graph of the MC.\\
The paper is structured as follows: Section \ref{Main} contains the main definitions  and results, whereas the other sections gather the proofs. All the results are first established in a simple example, called the \emph{Renewal Markov Chain}. They are then extended to the general case.\\ Section \ref{Renewal} introduces this simple example of null recurrent Markov Chain and the \textit{Renewal Bridge Graph} constructed from this MC. \\ Section \ref{S2} constructs a random graph on $\mathbb{Z}$, called the \emph{Renewal Eternal Family Tree (Forest)}. It will be shown that this random graph can be connected and form a tree or disconnected and form a forest. For the definition of Eternal Family Trees, see \cite{Baccelli2017}. Then other properties of the Renewal EFT, such as its unimodularity, are proved. \\ In Section \ref{S4}, a coupling between the Renewal Bridge Graph and the Renewal EFT is defined. This coupling helps one studying the properties of the Renewal Bridge Graph using the properties of the Renewal EFT.\\
Section  \ref{S5} considers the general null recurrent Markov Chain. It is shown that the structural properties established for the simple example hold for the general Bridge Graphs.\\ Section \ref{Null} considers the Taboo and Potential Dynamics on the Bridge Graph and studies their properties. The Taboo PP, is being introduced, which strongly relates to the unique invariant measure of the null recurrent Markov Chain (Theorem \ref{T1}). The constructibility of the Taboo and Potential PPs is also discussed. It is shown that these random measures are locally finitely constructible in the sense that the mass of the measure at each point only depends on an a.s. finite subgraph of the Bridge Graph. Nevertheless, it does not mean that it is always algorithmically constructible, and that one can find this finite subgraph effectively. This algorithmic constructibility holds nonetheless in the case where $\{{X}_{t}\}_{t \in \mathbb{N}}$ is \emph{monotone.} It means that in this case, one can perfectly sample from the Taboo and Potential PPs. A concrete example pertaining to queuing theory is also discussed in detail in Section \ref{sec:loy}. The $GI/GI/1$ queue 
allows one to illustrate the meaning and the practical interest of these two point processes. \\ 
Section \ref{Possitive}, considers the properties of the two point processes when the MC is positive recurrent. For the Taboo PP, the connection between perfect sampling in the CFTP sense, and the one obtained using the definition of the TPP is discussed. The properties of the Potential PP in the positive recurrent case are also considered. This section also gives some results about the properties of these two dynamics in more general state spaces. Instead of considering these dynamics on the Bridge Graph, they are regarded as Markov Chains on the space of the random measures on $S$. 
\section{Main Definitions and Results}
\label{Main}
Consider a Markov Chain $X= \{X_{t}\}_{t\in\mathbb{N}}$ defined on a probability space $(\Omega, \mathcal{F}, \mathbb{P})$ with a countable state space $\mathcal{S}$ and transition probabilities $P=(p_{x,y})_{x,y \in S}$. As mentioned in the introduction, two different dynamics are considered on the random counting measures or point processes with multiplicity on $\mathcal{S}$, satisfying \eqref{003}.\\
The first dynamics is the \textbf{Taboo Dynamics}, denoted by $H^{T}$ which is defined with respect to a reference point $s^{*} \in \mathcal{S}$. 
It is defined by ${M}_{t+1}^{T}=H^{T}({M}_{t}, \xi_{t})$, with, for each $x , y \in \mathcal{S}$,
\begin{equation}
\label{004}
{M}_{t+1}^{T}(y)=\begin{cases}
    \sum _{x \in \mathcal{S}}{M}_{t}^{T}(x) \mathbbm{1}_{\{ h(x,\xi_{t}^{x})=y\}}, & \quad y \ne s^{*}\\
    1, & \quad y=s^{*}.
  \end{cases}
\end{equation}
This dynamics constructs the random measure at time $t+1$ from the random measure at time $t$. It sends some mass from each state $x$ to state $y$ with rule $h$ while adding up the masses sent to the same state $y$. It ignores all the masses that enter $s^{*}$ at time $t+1$ and puts mass $1$ at this point.\\
The second dynamics is the \textbf{Potential Dynamics}. It is denoted by $H^{P}$. One has ${M}_{t+1}^{P}=H^{P}({M}_{t}, \xi_{t})$, with, for each $x , y \in \mathcal{S}$,   
\begin{equation}
{M}_{t+1}^{P}(y)=\begin{cases}
    \sum _{x \in \mathcal{S}}{M}_{t}^{P}(x) \mathbbm{1}_{\{ h(x,\xi_{t}^{x})=y\}}, & \quad  y \ne s^{*}\\
 \sum _{x \in \mathcal{S}}{M}_{t}^{P}(x) \mathbbm{1}_{\{ h(x,\xi_{t}^{x})=y\}} + 1 & \quad  y = s^{*}.
  \end{cases}
\end{equation}
As the Taboo Dynamics, there is a reference point $s^{*}\in \mathcal{S}$. For constructing the random measure at time $t+1$ from the random measure at time $t$, the Potential Dynamics sends some mass from each state $x$ to state $y$ with rule $h$ while adding up the masses sent to the same state $y$; in addition it adds mass one at point $s^{*}$. \\
The difference between the Taboo Dynamics and the Potential Dynamics is that the former always puts mass one at $s^{*}$ and does so by deleting the masses arriving at this point and adds mass one at $s^{*}$. In contrast, the Potential Dynamics just adds mass one at $s^{*}$. \\
The update rules of these two dynamics are related to the Doeblin coupling of the MC. Therefore, the main tool that will be leveraged to study these dynamics is the Doeblin Graph and its subgraph, the Bridge Graph. \\
As already mentioned in the introduction, the \textbf{Doeblin Graph} is a random graph constructed from $\{\xi_{x}^{t}\}$. The vertices of the Doeblin Graph are $\Sigma= \mathbb{Z} \times \mathcal{S}$. The first component of the vertices is considered as time, and hence the horizontal axis will be referred to as the \textbf{time axis}. The second component corresponds to the state of the vertices, and hence the vertical axis will be referred to as the \textbf{state axis}. There is an identically distributed and independent source of randomness $\{\xi_{t}^{x}, x\in \mathcal{S}\}_{t\in \mathbb{Z}}$, with $\xi_{t}^{x} \in \Xi=[0,1]$ such that $\{\xi_{t}^{x}\}$ is independent of the initial distribution of $\{X_{t}\}_{t\in\mathbb{N}}$. The function $h:\mathcal{S} \times \Xi \to \mathcal{S}$ defines the transitions between the states of $\mathcal{S}$. In addition, $h$ satisfies $P(h(x,\xi_{t}^{x})=y)=p_{x,y}$ for all $x,y \in \mathcal{S}$. The edges of the Doeblin Graph, $D$, are directed edges which are defined from a vertex $(x,t)$ at time $t$, to a vertex at time $t+1$, through the random map
\begin{equation}
(t,x) 	\mapsto (t+1,h(x,\xi_{t}^{x})).
\end{equation}
Consider the subgraph of the Doeblin Graph of $X$ that contains those vertices which are in the union of the trajectories starting from all $(t,s^{*})$, $t \in \mathbb{Z}$. This gives a subgraph of the Doeblin Graph called the \textbf{(Doeblin) Bridge Graph, $B_{X}$}. Here $s^{*}$ is a fixed arbitrary state in $\mathcal{S}$. 
The properties of the positive recurrent Bridge Graph were studied in \cite{Baccelli2018}. The most important properties are the fact that this graph is unimodularizable and the existence of a unique bi-recurrent path $\{ \beta _{t} \}_{t \in \mathbb{Z}}$. The bi-recurrent path is a path such that the number of times that it meets each state $x$ in $\mathcal{S}$, in both positive and negative times, is infinite a.s. Based on this definition, when a random path is bi-recurrent, it is bi-infinite. The existence of a bi-recurrent path is established from the unimodularizability of the positive recurrent Bridge Graph in the sense of \cite{Aldous}. The importance of the bi-recurrent path is that the vertices belonging to this path are distributed as the stationary distribution of the MC from which the Bridge Graph was constructed. So each vertex in this path can be considered as a perfect sample of the stationary distribution of the MC. \\
\begin{remark}
\label{RT}
In the definition of the Doeblin Graph and the dynamics defined in \eqref{003}, the same source of randomness $\{\xi_{t}^{x}, x\in \mathcal{S}\}_{t\in \mathbb{Z}}$ is used. This sequence is always considered to be independent in $t$ but not necessarily in $x$. i.e., at each time $t$, the random variables $\{\xi_{t}^{x} \}$ can be coupled. When the random variables  $\{\xi_{t}^{x} \}$ are independent in both $t$ and $x$, this sequence will be called \emph{totally independent}.
\end{remark}
\begin{example}
\label{MCLRW}
Consider the \emph{lazy random walk}, W, defined on the state space $\mathcal{S}=\mathbb{Z}$, with the following transition probabilities: the walk stays at the current state, $i$, with probability $1/3$, and moves to each neighbor of $i$, i.e.,  $i+1$ or $i-1$, at random with probability $1/3$. Then, one can consider the Doeblin Graph of $W$, constructed from the driving sequence $\{\xi_{t}^{y}, y\in \mathcal{ \mathbb{Z}}\}_{t\in \mathbb{Z}}$, and the transition function $h$, where $\{\xi_{t}^{y}, y\in \mathcal{ \mathbb{Z}}\}$ are maximally coupled for a given $t$, i.e., at each time $t$, for all $i$, $\xi_{t}^{i}=\xi_{t}^{0}$, and for all $i$ the transition $h(i,\xi_{t}^{i})=h(i,\xi_{t}^{0})=i+h(0,\xi_{t}^{0})$. In this example, one can show that the Bridge Graph of this lazy random walk coincides with its Doeblin Graph (see Figure \ref{MCLRWF}).
\end{example}
\begin{figure}[h]
\includegraphics[width=\textwidth]{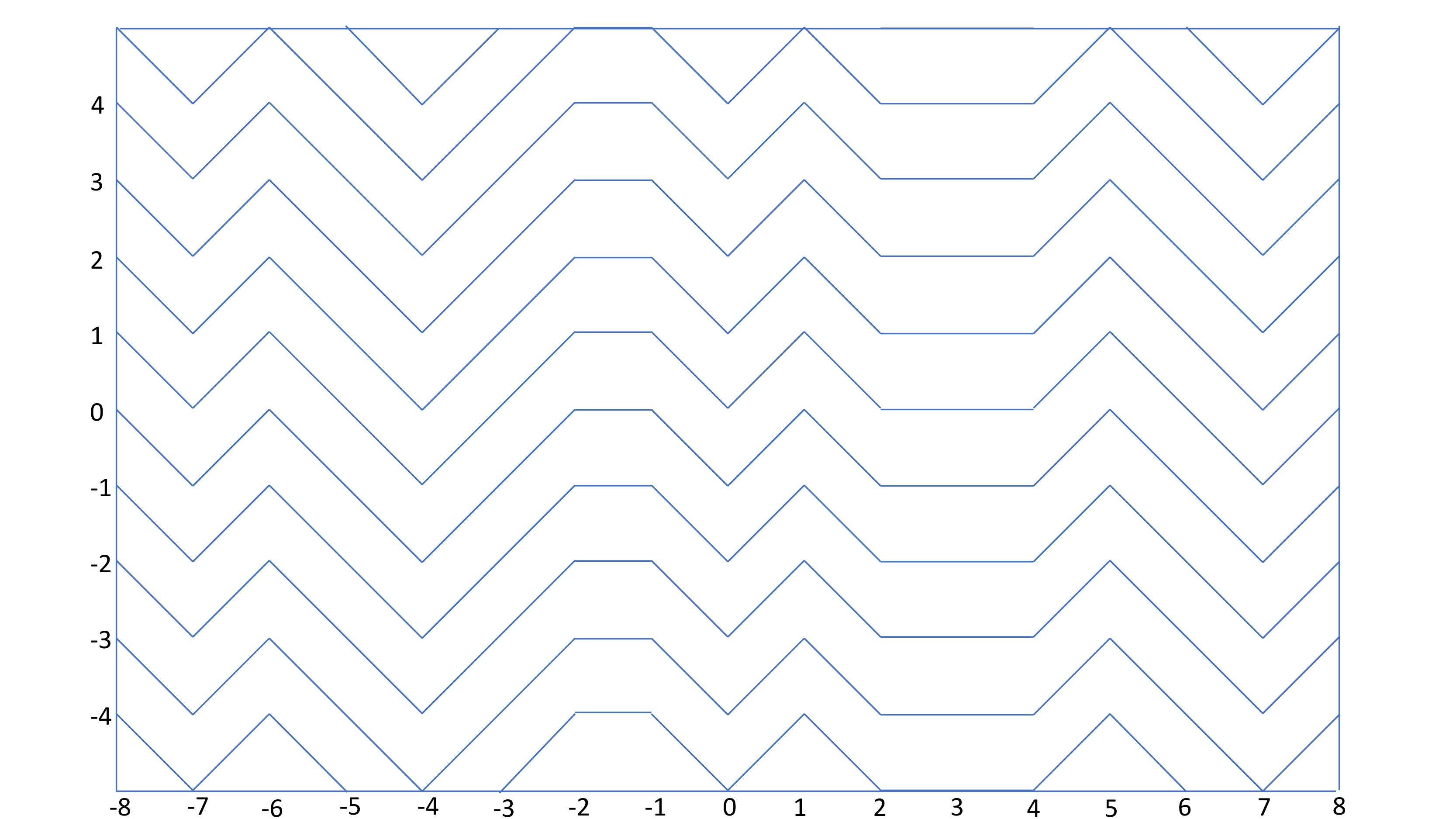}
\caption{The maximally coupled Doeblin Graph and Bridge Graph of lazy random walk}
\label{MCLRWF}
\centering
\end{figure}
In this paper, it will be shown that the null recurrent Bridge Graph has the following properties: 
\begin{proposition}
\label{Pr13}
The null recurrent Bridge Graph is either a tree or a forest, and both cases can happen.
\end{proposition}
When the null recurrent Bridge Graph is connected, or equivalently it is a tree the following property holds:
\begin{proposition}
\label{Pr11}
If the null recurrent Bridge Graph is a tree, it has no bi-infinite path. 
\end{proposition}
A similar result is valid in the non-connected case with an additional condition:
\begin{proposition}
\label{Prnew}
Consider a null recurrent Bridge Graph, $B_{X}$, which may not be connected. Suppose the Bridge Graph satisfies this condition that for all $(t_{1},s_{1})$ and $(t_{2},s_{2}) \in \mathbb{Z} \times \mathcal{S}$, the paths passing through $(t_{1},s_{1})$ and $(t_{2},s_{2})$ meet each other with positive probability in finite time. Then, the Bridge Graph has no bi-infinite path. 
\end{proposition}
The main result about unimodularity is: 
\begin{proposition}
\label{Pr12}
When the null recurrent Bridge Graph is a tree, it is not unimodularizable, in general.
\end{proposition}
Consider all the vertices in the intersection of the Bridge Graph and the zero timeline, i.e., those on the state axis. This random set  will be referred to as the \textbf{$S$-set}. The properties of the $S$-set in the null recurrent Bridge Graph are studied in Section \ref{S5}. \\
Two multiplicities for a point in the $S$-set are now defined.
One can look at each vertex in the Bridge Graph (or any directed graph) as an individual.
Moreover, by following the outgoing edge, go from each vertex to its parent vertex.
In the Bridge Graph, one can consider the descendants of a vertex that lie on the time axis, i.e.,
belonging to $\mathbb{Z} \times \{s^{*}\}$. Descendants of this type are referred to as 
\textbf{*-descendants}.  

\begin{definition}
\label{D06}
The \textbf{Taboo multiplicity} of a point in the $S$-set of a Bridge Graph is the number of its *-descendants such that the path from this descendant to the $S$-set does not visit state $s^{*}$ before time zero. See Figure~\ref{Pic1}. Note that by definition, the Taboo multiplicity is positive at all points of the $S$-set.
\end{definition}
\begin{remark}
Note that the order of the generations is not consistent with the time direction, as ancestors live in the future, and these notions should not be mixed up. 
\end{remark}
\begin{figure}[h]
\includegraphics[width=\textwidth]{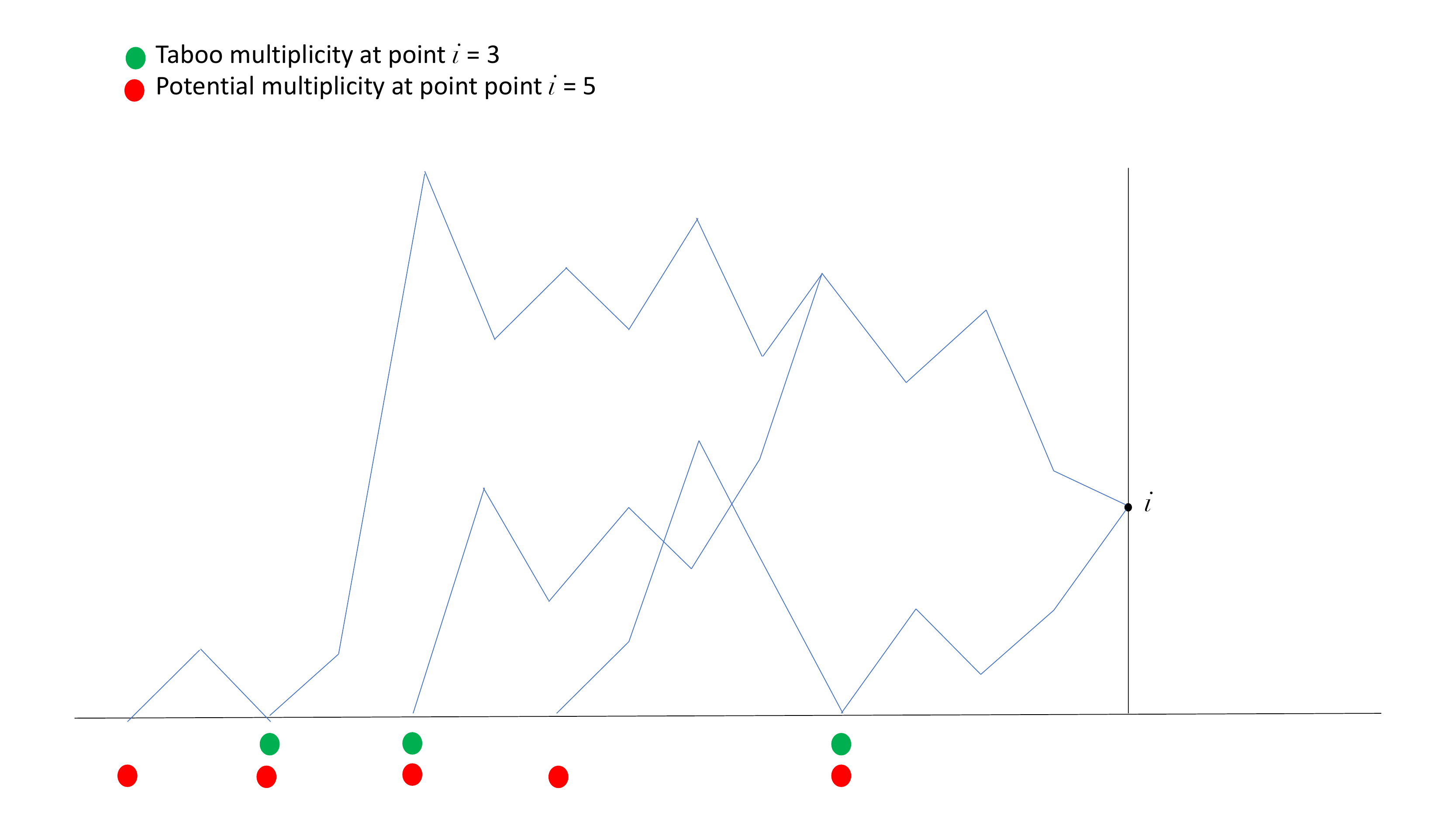}
\caption{The Taboo multiplicity and the Potential multiplicity}
\label{Pic1}
\centering
\end{figure}
So far, it has has been shown that the Taboo multiplicity of any vertex in the Bridge Graph is a.s. finite, so the definition of Taboo multiplicity gives a locally finite random measure whose support is the $S$-set. This random measure is called the \textbf{Taboo Point Process (Taboo PP)} and is denoted by $\tau$. Below, $\tau_{t}(j)$ denotes the random mass (multiplicity) that the Taboo PP puts on $j$ at time $t$. \\
\begin{theorem}
\label{Propo1}
The Taboo PP, $\tau$, is a steady state of the Taboo Dynamics, 
\begin{equation}
\tau_{t+1}=H^{T}(\tau_{t}, \xi_{t}).
\end{equation}
\end{theorem}
The following theorem shows that there is a relation between the Taboo PP in the null recurrent Bridge Graph and the stationary measure of the null recurrent Markov Chain:
\begin{theorem}
\label{T2}
Let $\{X_{t}\}_{t\in \mathbb{Z}}$ be an aperiodic and recurrent MC,  and $B_{X}$ be its associated  Bridge Graph with the driving sequence $\{ \xi_{t}^{x} \}_{t\in \mathbb{Z},x\in S}$. Then $\mathbb{E}[\tau_{t}(i)] $, the mean measure of the Taboo PP at points, does not depend $t$, nor on the coupling of $\{ \xi_{t}^{x} \}_{t\in \mathbb{Z},x\in S}$ in $x$, and it is equal to the stationary measure of that point in the Markov Chain $\{X_{n}\}$. That is, 
\begin{equation}
\label{5-80}
\mathbb{E}[\tau_{t}(i)]=\sigma(i), \quad \forall i \in \mathcal{S},
\end{equation}
where $\sigma$ is the invariant measure of the Markov Chain $\{X_{n}\}$ and $\sigma(s^{*})=1$.\\
\end{theorem}
The second multiplicity that will be considered for a point in the $S$-set is the ``Potential multiplicity'':
 \begin{definition}
\label{D07}
The \textbf{Potential multiplicity} of a point in the $S$-set is the number of all its *-descendants in the Bridge Graph. See Figure \ref{Pic1}.
\end{definition}
The Potential multiplicity on the $S$-set gives a random measure with support the $S$-set itself. This random measure is called the \textbf{Potential Point Process (Potential PP)} and denoted by $\pi$.  Again, $\pi_{t}(j)$ denotes the random mass that the Potential PP puts on $j$ at time $t$. Proposition \ref{Pr16} shows that the Potential multiplicity of the vertices in the null recurrent Bridge Graph is a.s. finite. In addition the following result holds: \\
\begin{theorem}
\label{Pr18}
In the null recurrent case, the Potential PP, $\pi$, is a steady state of the Potential Dynamics,
\begin{equation}
\pi_{t+1}=H^{P}(\pi_{t}, \xi_{t}).
\end{equation}
\end{theorem}
The relation between the Potential PP of a null recurrent Bridge Graph and the associated  MC is summarized in the following theorem. 
\begin{theorem}
\label{T3}
Consider a null recurrent Markov Chain $\{X_{n} \}$ and its associated Bridge Graph $B_{X}$. The mean measure of the Potential PP is equal to the entries of a row in the  potential matrix of the Markov Chain $\{X_{n}\}$. So,
\begin{equation}
\label{5-8}
\mathbb{E}[\pi_{t}(i)]=\infty, \quad  \forall i \in \mathcal{S}.
\end{equation}
\end{theorem}
\begin{remark}
Equation (\ref{5-8}) remains valid in the positive recurrent case. Also, in the transient Bridge Graph, it can be shown that the potential multiplicities are such that their means are equal to the entries of the classical potential matrix for MCs \cite{Pierre}. This is why the multiplicity, the associated point process and the dynamics are called ``potential''. In the positive recurrent case, since the value of the potential multiplicity, in one state is infinite, it is not a locally finite measure. So Theorem \ref{Pr18} does not hold in the positive recurrent case. 
\end{remark}
\section{Renewal Bridge Graph}
\label{Renewal}
This section first introduces the Renewal Markov Chain, a simple example of recurrent Markov Chain, which may be positive or null recurrent. After that, the Renewal Bridge Graph, which is the Bridge Graph constructed from the Renewal Markov Chain, is introduced. Before going through the proof of the properties of the general null recurrent Bridge Graph, the proofs are first established in this particular example. \\
Consider random variable $\eta$ on $\mathbb{N^{*}}$ \footnote{In this paper, $\mathbb{N^{*}}$ denotes the natural numbers without zero and $\mathbb{N}$ natural numbers with zero.} with distribution 
\begin{equation}
\label{5-0-0}
\Lambda=\{p_{k},k \in \mathbb{N}\}; p_{k}=\mathbb{P}(\eta=k).
\end{equation}
Suppose that $\eta$ is such that if 
\begin{equation*}
\label{A}
A=\{n+1 \in \mathbb{N}; p_{n}>0\},
\end{equation*}
then $A$ is infinite, and the greatest common divisor of $A$ is equal to $1$. Using this random variable, one can define the following Markov Chain:
\begin{definition}[Renewal Markov Chain]
\label{5-0-1}
Consider the following transition probabilities on the non-negative integers: for $i \ne 0$ 
\begin{equation}
\label{5-14}
p_{ij}=
\begin{cases}
1& \text{if}\ j=i-1, \\
0& \text{otherwise,}\ 
\end{cases}
\end{equation}
and for $i = 0$ 
\begin{equation}
\label{5-5}
p_{0j}= p_{j+1},
\end{equation}
where the $p_{j}$s are the probabilities of the random variable $\eta$ defined in \eqref{5-0-0}. This Markov Chain is called the \textbf{Renewal Markov Chain}. The assumptions that $A$ is infinite and $gcd(A)=1$ make the Renewal MC irreducible and aperiodic. Starting from $0$, it a.s. returns to this point. So point $0$ is recurrent, and thus the Markov Chain is recurrent. Let $T_{0}^{+}$ be the first return time to $0$ starting from $0$. Then
\begin{equation}
\label{5-4}
\mathbb{E} [T_{0}^{+}] = \sum _{i} \mathbb{E} [T_{0}^{+} | \text{first jump is i}] p(\text{first jump is $i$})=\sum _{i} (i+1) . p_{0(i+1)}= \mathbb{E} [\eta] . 
\end{equation}
So if $\mathbb{E} [\eta]=\infty$, this Markov Chain is null recurrent, and hence in this case it is called the  \textbf{null recurrent Renewal Markov Chain}.\\
\end{definition}
The Doeblin Graph of the Renewal Markov Chain is as follows: the set of vertices of this random graph is $\Sigma=\mathbb{Z} \times \mathbb{N}$, and the driving sequence is $\xi_{t}^{n}, n\in \mathbb{N}$. For $i \ne 0$, vertex $(t,i)$ has a single outgoing edge which goes to the vertex $(t+1,h(i,\xi_{t}^{i}))$ with $h(i,\xi_{t}^{i})=i-1$ a.s. For $i=0$,  the outgoing edge from $(t,0)$  goes to vertex $(t+1,h(0,\xi_{t}^{0}))$  with $\mathbb{P}(h(i,\xi_{t}^{0})=j)=p_{0j}$ in \eqref{5-5}. The vertices of this graph, which are formed by the union of the trajectories starting from $(t,0)$, $t \in \mathbb{Z}$, form the Bridge graph with $s^{*}=0$. This graph is called the \textbf{Renewal Bridge Graph}. 

\section{Renewal Eternal Family Tree} 
\label{S2}
The Renewal Eternal Family Tree (EFT) is a random graph defined from Renewal Bridge Graph. 
\begin{definition}
\label{D01}
Consider the directed random graph $G^{\eta}=(V,E)$ on $\mathbb{Z}$, with vertices $V=\mathbb{Z}$. The set of edges, $E$, is as follows: at each vertex $i$, there is an edge to vertex $i+\eta_{i}$, where the random variables $\{\eta_{i}\}_{i\in \mathbb{Z}}$ are i.i.d. with $\eta_{i} \sim \eta$ defined in \eqref{5-0-0}. One can verify that this graph has no loops, and hence it is either a tree or a forest. If this graph is a tree, it has all the properties of an Eternal Family Tree as defined in \cite{Baccelli2018}. So $G^{\eta}$ is called  the \textbf{Renewal Eternal Family Forest (Renewal EFF)}. In the connected case, it is referred to as the \textbf{Renewal Eternal Family Tree (Renewal EFT)}.
\end{definition}
Below, it is assumed that $\mathbb{E}(\eta)= \infty$. Proposition \ref{T1} shows that both a tree and forest can arise. This proposition considers a specific distribution for $\eta$, satisfying the infinite mean property.  For this particular distribution, under certain conditions, $G^{\eta}$ is an EFT, and under others is a forest. 
\begin{proposition}
\label{T1}
Suppose that $\eta$ has following probability distribution 
\begin{equation}
\label{4-0}
P(\eta=j)= q_{j}=\frac {c_{1}}{j^{\alpha +1}} , \quad  0<\alpha \leq 1, j \geq 1,
\end{equation}
which gives the following tail distribution :
\begin{equation}
\label{4-0-1}
P(\eta>j)=\frac {c_{2}}{j ^{\alpha}} , \quad 0<\alpha \leq 1, j \geq 1, 
\end{equation}
where $c_{1}$ and $c_{2}$ are normalizing constants. Then the random graph constructed in Definition \ref{D01} is a.s. a Renewal EFT when $\alpha \geq \frac{1}{2}$ and a.s. a forest when $\alpha < \frac{1}{2}$. 
\end{proposition}
The proof of Proposition \ref{T1} is provided in Subsection \ref{Proof} of the appendix. 
For the remainder of the document, it is assumed that $\eta$ satisfies $\mathbb{E}(\eta)= \infty$, unless mentioned otherwise. 
\begin{remark}
The Renewal EFF is not limited to distribution \eqref{4-0}. This distribution is considered only for showing that both the Renewal EFT and the Renewal EFF exist when $\eta$ has an infinite mean.
\end{remark}
\subsection{Properties of the Renewal EFF}
Here are some properties of Renewal EFF to be used later. Proposition \ref{Pr1} studies the unimodular property of the Renewal EFF. For the definition and some examples of unimodular random networks, see \cite{Aldous}.
\begin{proposition} 
\label{Pr1}
The Renewal EFF is a unimodular network. 
\end {proposition}
\begin{proof}
Let $(G,o)$ is be the deterministic graph with vertices $V=\mathbb{Z}$, and edge set, $E=\{(n,n+1); \forall n \in \mathbb{Z} \}$ that is rooted at  $0$. This is a unimodular graph.\\
For all unimodular networks, it is possible to enrich vertices and edges with i.i.d. marks and preserve unimodularity (See \cite{Aldous}). Since the Renewal EFF is a random graph constructed from i.i.d. marks added to $(G,0)$, it is a unimodular network. 
\end{proof}
The following proposition requires some more properties of random networks. Here is a brief review of these properties. For more details on these concepts, see \cite{Baccelli2017}. 
One can define a  \textbf{vertex shift} on any network $G=G(V,E)$. A vertex shift  $f_{G}$ is a function on the vertices,  $f_{G}: V \to V$, such that $f_{G}$ commutes with network isomorphisms, and such that the function $[G,o,s] \mapsto 1_{f_{G}(o)=s}$ is measurable on $\mathcal{G}_{**}$. Here, $\mathcal{G}_{**}$ denotes the set of isomorphism classes of rooted, connected, and locally finite networks with a pair of distinguished vertices. The \textbf{$f$-graph} of $G$ is the graph $G ^{f}= (V,E^{f})$, with the set of vertices $V$ and directed edges $E^{f} (G)= \{(x,f(x)); x \in V \}$. Each connected component of the $f$-graph is called an  \textbf{$f$-component} of the graph. Consider the following equivalence relation on $V$:
\begin{equation}
\label{foil}
x \sim _{f} y \iff 	\exists n \in \mathbb{N}; \quad f_{G} ^{n} (x) = f_{G} ^{n} (y).
\end{equation}
Every equivalence class of this equivalence relation is called a \textbf{foil}. The \textbf{Foil Classification in Unimodular Networks Theorem} (Theorem 3.10 in \cite{Baccelli2017}) states that in a unimodular network $(G,o)$, for all vertex shifts $f$, each connected component, $C$, of its $f$-graph belongs to one of the following classes:
\begin{enumerate}[label=\roman*]
\item Class $\mathcal{F/F}$:  $C$ and all its foils are finite, and there is a unique cycle in $C$.
\item Class $\mathcal{ I/F}$: $C$ is infinite but all its foils are finite, there are no cycles in
C, and there is a unique bi-infinite path in $C$.
\item Class $\mathcal{I}/\mathcal{I}$: $C$ is infinite, all its foils are infinite, and there are no cycles or bi-infinite paths in $C$.
\end{enumerate}
\begin{proposition}
\label{Pr2}
Let $G^{\eta}=G^{\eta}(V,E)$ be a Renewal EFF. Then each connected component of $G^{\eta}$ is $\mathcal{I}/\mathcal{I}$ in the sense of the foil classification theorem of unimodular networks.
\end {proposition}
\begin{proof}
Consider the vertex shift $f$ on $G^{\eta}$, which maps each vertex to its right adjacent vertex. Each connected component of the Renewal EFF is an infinite tree, so it is either in the $\mathcal{I}/\mathcal{I}$ class or the $\mathcal{I}/\mathcal{F}$ class of the foil classification theorem. First, suppose that $G^{\eta}$ is connected and that it is $\mathcal{I}/\mathcal{F}$. It follows that there is a unique bi-infinite $f$-path, $\mathcal{P}$, in this component (see Theorem $3.10$ in \cite{Baccelli2017}). Since $\mathcal{P}$ is unique, it is distinguishable in the whole graph. So it is a covariant subgraph of $G^{\eta}$. Using Lemma $2.8$ of \cite{Baccelli2017}, $\mathcal{P}$ has a positive density in $\mathbb{Z}$. In the sense that $\mathbb{P}(0 \in V_{\mathcal{P}})>0$, where $V_{\mathcal{P}}$ is the set of vertices of $\mathcal{P}$. Consider the measurable function $g$ defined as follows: $g[G^{\eta},x,y]\equiv0$, if there is no bi-infinite path in $G^{\eta}$. When the bi -infinite path $\mathcal{P}$ exists:
\begin{itemize}
\item $g[G^{\eta},x,y]=1$, if $x,y$ are two consecutive vertices in $\mathcal{P}$ such that $x<y$,
\item $g[G^{\eta},x,y]=1$ , if $x \notin V_{\mathcal{P}}$, and $y$ is the nearest vertex to the left of $x$ that belongs to $V_{\mathcal{P}}$, and
\item $g[G^{\eta},x,y]=0$, otherwise.
\end{itemize}
Using the mass transport principle, one can write: 
\begin {equation}
\label{Pr2-1}
\mathbb{E}\left[ \sum _{x\in\mathbb{Z}} g[G^{\eta},0,x]\right]=  \mathbb{E}\left[ \sum _{x\in\mathbb{Z}} g[G^{\eta},x,0]\right].
\end{equation}
The left-hand side of equation \eqref{Pr2-1} is equal to the probability of the existence of $\mathcal{P}$ in $G^{\eta}$, and the right-hand side of the equation is equal to 
\begin {equation}
\label{Pr2-2}
\mathbb{P}(0 \in V_{\mathcal{P}}). \mathbb{E}[\text{ number of the vertices between $0$ and its right neighbor in $\mathcal{P}$} | 0 \in V_{\mathcal{P}}].
\end{equation}
Since $\mathbb{P}(0 \in V_{\mathcal{P}})>0$, the equality of \eqref{Pr2-1} and \eqref{Pr2-2} gives that the expectation of the number of the vertices between $0$ and its right vertex in $V_{\mathcal{P}}$ given that $0 \in V_{\mathcal{P}}$ is finite. \\
On the other hand, note that the existence of a bi-infinite path is a property of the left-hand side of the graph. In the sense that if there exists a path that comes from $-\infty$ and reaches zero, it is bi-infinite. The path on the right-hand side of $0$ is ``fresh'', and the distribution of the length of $\mathcal{P}$ edges on this side is the same as $\eta$. So the distance between $0$ and its right neighbor in $\mathcal{P}$ has an infinite mean, while it is shown above that this expectation is finite. Thus this path does not exist. Hence, the tree belongs to the $\mathcal{I}/\mathcal{I}$  class.\\
Consider now the case where $G^{\eta}$ is not connected. Suppose only one bi-infinite path exists in the graph. Then this path is again a covariant subgraph of $G^{\eta}$, i.e., with positive probability, zero belongs to this path, and the same argument as in the EFT case shows that it is impossible. So either there is no bi-infinite path in the graph, or there is more than one bi-infinite path. Suppose that the latter case happens. The variables $\eta_{i}, i<0$ determine the number of bi-infinite paths in the EFF. Let $\mathcal{P}_{1}$ and $\mathcal{P}_{2}$ be two bi-infinite paths that come from $-\infty$ and reach time zero. Since after $0$, these two paths do not depend on the past, and the $gcd$ of $A$ in \eqref{A} is equal to one, these two paths meet each other with positive probability. So there is more than one bi-infinite path in one connected component of the EFF with positive probability, which is impossible due to the foil classification theorem of unimodular networks.
\end{proof}

Rephrased in terms of the classification of unimodular EFTs, the last result
complements the known fact that 
the renewal EFT is ${\mathcal{I/F}}$ in the
case where the renewal distribution has finite mean, by showing 
that it is ${\mathcal{I/I}}$ when this mean is infinite.

\section{Properties of the Renewal Bridge Graph}
\label{S4}
\subsection{Basic Properties}
In the Renewal MC, the transition probabilities from zero are the same as the jumps distribution in the Renewal EFT (EFF). Using this, one can define a coupling between the Renewal EFT (EFF) and the Renewal Bridge Graph. 
\begin{definition}
In the Renewal Bridge Graph, consider $e_{t}$, the outgoing edge
\footnote{In the Bridge Graph (and in the EFF), there is a  natural direction for the edges from
time $t$ to time $t + 1$. With this direction, each edge has a beginning vertex and an end vertex.}
at vertex $(t,s^{*})$. Let $(t+1, s')$ be the end of edge $e_{t}$.
Then the jump at time $t$ is defined by $|s'-s^{*}|$. 
\end{definition}
Let $ \{\eta_{i}\}_{i \in \mathbb{Z}}$ be the length of the outgoing edge at vertex $i$ in the Renewal EFT, and $ \{\eta'_{i}\}_{i \in \mathbb{Z}}$ be the jump at time $i$ in the Renewal Bridge Graph. Then
\begin{equation*}
\eta_{i} \sim \eta'_{i}+1        \quad \forall  i .
\end{equation*}
The coupling between $(\eta_{i},\eta'_{i})$ is defined by taking 
\begin{equation}
\label{Co}
{\eta}_{i} ={\eta'}_{i}+1 \quad \forall i \quad \text{a.s.}
\end{equation}
\begin{figure}[h]
\includegraphics[width=\textwidth]{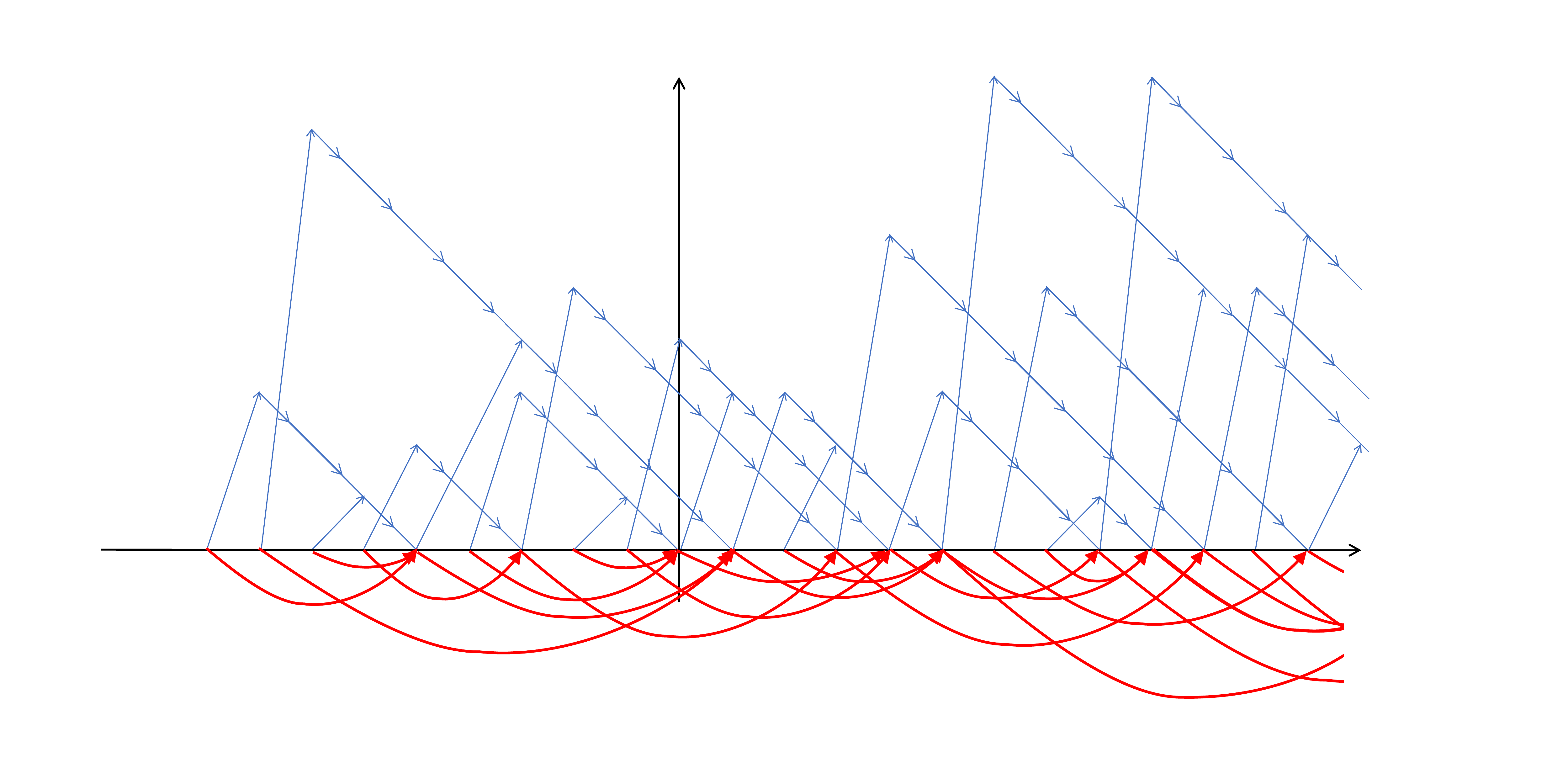}
\caption{Coupling between Renewal EFT and Renewal Bridge Graph}
\label{Pic2}
\centering
\end{figure}
\begin{proposition}
\label{Pr3}
The null recurrent Renewal Bridge Graph is either a tree or a forest. Both cases can happen, i.e., there are examples where the Renewal Bridge Graph is a tree and examples where it is a forest.
\end{proposition}
\begin{proof}
In the Bridge Graph, there is only one outgoing edge from each vertex. Also, the edges are just going forward in time. So there is no cycle. Hence the Bridge Graph is either a tree or a forest. It remains to show both cases are possible. This is because, given the coupling \eqref{Co}, by the following argument, the connectedness of the Renewal Bridge Graph and the Renewal EFT(EFF) are equivalent.\\
Suppose that in the Renewal EFT, there is an edge from vertex $t$ to vertex $t'=t+j$, where $j$ is the value of $\eta_{t}$. Correspondingly, using the coupling defined in \eqref{Co}, in the Renewal Bridge Graph, there is an edge $\eta'_{i}=\eta_{i}-1=j-1$ between the vertices $(t,0)$ and $(t+1,j-1)$. Due to the construction of the Renewal Bridge Graph, it has a decreasing path from vertex $(t+1,j-1)$ to vertex $(t+j,0)$. It means that there is a path in the Bridge Graph starting from vertex $(t,0)$ in the time axis and back to the time axis for the first time again at vertex $(t+j,0)$. \\
So if two paths in the Renewal EFT (EFF) starting from two different vertices in the Renewal EFT (EFF) meet each other at a given time, the paths starting from the corresponding vertices in the Renewal Bridge Graph meet each other and vice versa. See Figure \ref{f1}. Thus a Renewal Bridge graph is a tree if and only if its corresponding Renewal EFT is a tree.
Then the result follows from Proposition \ref{T1}.
 \end{proof}
 \begin{figure}
\caption{Equivalence of  connectedness of the Renewal Bridge Graph and Renewal EFT }
\includegraphics[width=\textwidth]{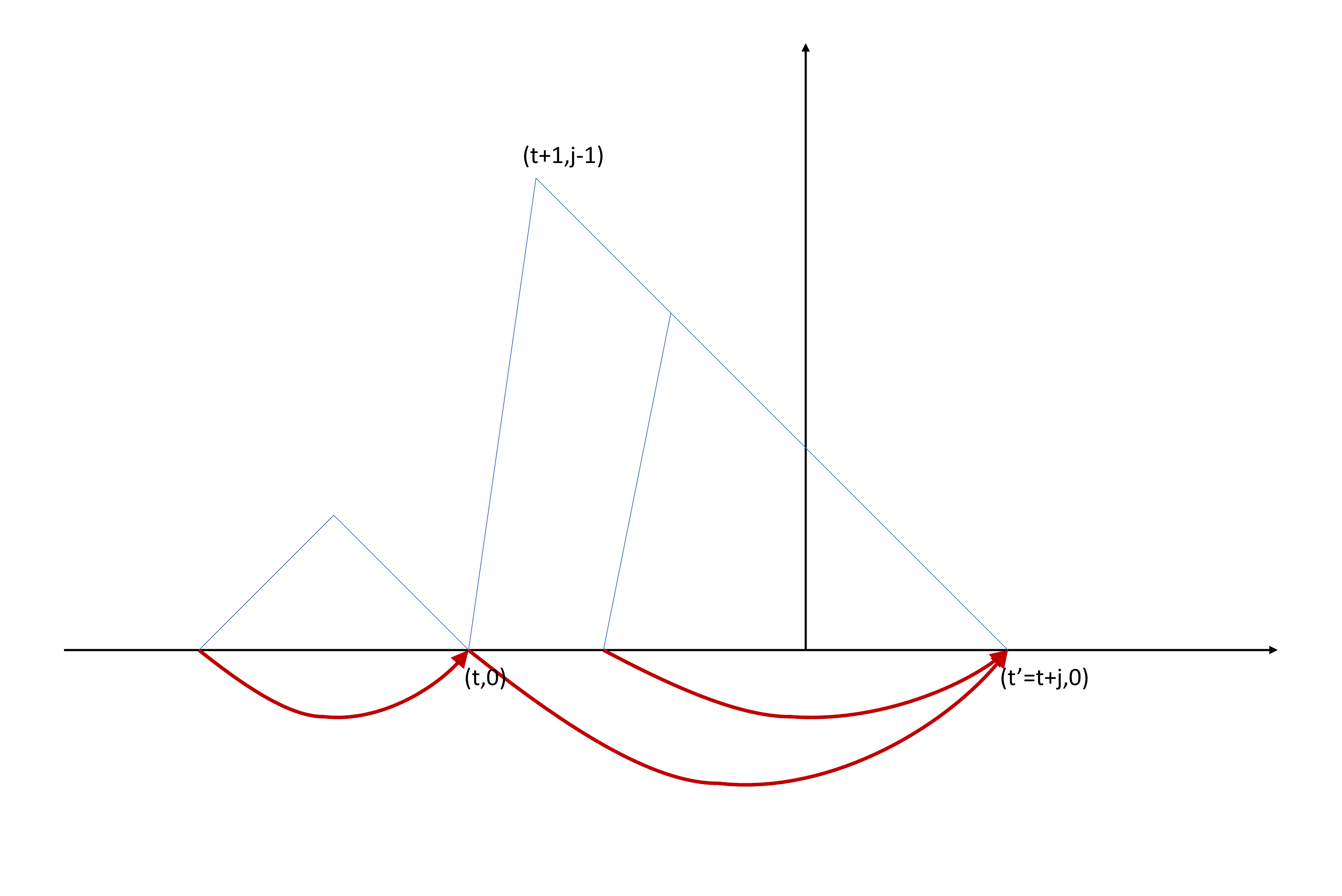}
\label{f1} 
\end{figure}
\begin{proposition}
\label{Pr3-4}
Every bi-infinite path, $\{\beta_{t}\}_{t\in\mathbb{Z}}$, in the Renewal Bridge Graph, $B_{X}$, corresponds to a bi-infinite path in its associated Renewal EFF. 
\end{proposition}
\begin{proof}
The proof consists in proving that, every bi-infinite path $\{\beta_{t}\}_{t\in\mathbb{Z}}$, in the Renewal Bridge Graph, is bi-recurrent, i.e., it meets the time axis in both the positive and negative parts a.s., infinitely many times. Since the MC $X$ is recurrent, every bi-infinite path meets the positive part of the time axis a.s. infinitely many times. So it is enough to show that it meets time axis infinitely many times in the negative part. \\
let  $T$ be an arbitrary element of the time axis. Consider the following set in the Renewal Bridge Graph:\\

$D=$\{$T-t; $ the path that starts from $(T-t,s^{*})$ goes back to the time axis for the first time at time $T$\}. \\

For a fix $t$, the probability that $T-t \in D$ equals the probability that at time $T-t$, the jump is equal to $t-1$, i.e. $p_{t-1}$. Since $\sum_{t=1}^{\infty}p_{t}= 1$, one can conclude,  using the Borel-Cantelli lemma, that the cardinality of $D$ is a.s. finite. On the other hand, suppose there exists an infinite path, $\{\beta_{t}\}_{t\leq T}$, in the Renewal Bridge Graph that comes from $(-\infty,+\infty)$ and reaches the time axis for the first time at $T$. Hence
\begin{equation}
\label{333-444}
\{\beta_{t}\}_{t\leq T}=\{(t,T-t)\}_{t\leq T}.
\end{equation}
Since $\{\beta_{t}\}_{t\leq T}$ is a path in the Renewal Bridge Graph, every vertex in this path has a backtrack to the time axis. It means that there exist infinitely many edges starting from the time axis and ending up at $\{\beta_{t}\}_{t\leq T}$. Note that \eqref{333-444} gives that the probability that this happen is equal to the probability that $D$ be infinite, which is equal to zero. So a.s., in the Renewal Bridge Graph, there is no bi-infinite path that comes from $(-\infty,+\infty)$ and reaches the time axis for the first time at some $T$. So every bi-infinite path in the Bridge graph is bi-recurrent. 
\end{proof}
\begin{proposition}
\label{Pr4}
The Renewal Bridge Graph has no bi-infinite path.
\end{proposition}
\begin{proof}
Proposition \ref{Pr3-4} states that every bi-infinite path in the Renewal Bridge Graph is bi-recurrent. So, if there is a bi-infinite path in the Renewal Bridge Graph, some vertices in this bi-infinite path have infinitely many descendants in the time axis. It means that correspondingly some vertices in the Renewal EFF have infinitely many descendants, which contradicts the fact that every connected component of EFF is $\mathcal{I/I}$, as shown by Proposition \ref{Pr2}.
\end{proof}
Some further definitions on unimodularizability of a random network are needed. The following definition borrowed from \cite{ALI}:\\
\begin{definition}
\label{D5-4}
Let $[G]$ be a non-rooted random network. A random rooted network $[G',o']$  is \textbf{unroot-equivalent} to $[G]$ if, by forgetting the root, the distribution of $[G']$ is identical to the distribution of [G]. A random network $[G]$ is said \textbf{unimodularizable} if there exists a unimodular random rooted network $[G',o']$ which is unrooted-equivalent to[G].\\
Similarly, one can define the notion of unroot-equivalence between two random rooted networks.
\end{definition}
\begin{proposition}
\label{Pr5}
The null recurrent Renewal Bridge Graph is not unimodularizable in general.
 \end{proposition}
 \begin{proof}
 The proof is similar to that of Proposition \ref{Pr12}, which will be presented in Section \ref{S5}. 
\end{proof}
In the Renewal Bridge Graph, the function that maps every vertex to its right adjacent vertex, in the next time, is a vertex shift in the sense of \cite{Baccelli2017}. As defined in Section \ref{S2}, considering this vertex shift, one can considers its foils in the Bridge Graph. 
\begin{proposition} 
\label{Pr6}
The foils of a connected components of the Renewal Bridge Graph are its intersections with vertical timelines. There are infinitely many foils in each connected component of the Bridge Graph, and the order of the foils is of type $\mathbb{Z}$.
 \end{proposition}
\begin{proof}
It will be shown in Proposition \ref{Pr14} that the same property holds in the Bridge Graphs constructed  by a general null recurrent Markov Chain. So the result is valid for the Renewal Bridge Graph as well. 
\end{proof}
\subsection{Properties of the $S$-set in the Renewal Bridge Graph}
\begin{definition}
Consider the intersection of the Bridge Graph with the zero timeline. This set is a random subset of the state space $\mathcal{S}$, referred to as the \textbf{$S$-set}. 
\end{definition}
In the Renewal Bridge Graph, suppose that vertex $(0,y)$ belongs to the $S$-set. Then, since the vertices of the Renewal Bridge Graph have a backtrack to a vertex in the time axis, if one goes backward in time, from vertex $(0,y)$, one eventually reaches a vertex in the time axis for the first time. This vertex is denoted by $\mathbf{(t^{-}_{y},0)}$. Also, by continuing the path that passes through the vertex $(0,y)$ forward in time, it will also reach a vertex on the time axis. Denote this vertex by $\mathbf{(t^{+}_{y},0)}$. Note that by the definition of the Renewal Markov Chain, $(t^{+}_{y},0)=(y,0)$. So for each vertex $y \ne 0$ in the $S$-set of the Renewal Bridge Graph, there is a path in the graph that starts from a vertex on the time axis before time zero and returns to the time axis, for the first time, after time zero. Correspondingly, under the coupling \eqref{Co}, there is an edge in the Renewal EFT that starts from vertex $t^{-}_{y}$ before time zero and ends at vertex $t^{+}_{y}$ after time zero. See Figure \ref{F-5-1}. 
\begin{figure}[h]
\caption{Correspondence between an edge that flying over zero in the Renewal EFT and a path that started before time zero and ended after time zero in the Renewal Bridge Graph. }
\includegraphics[scale=0.4]{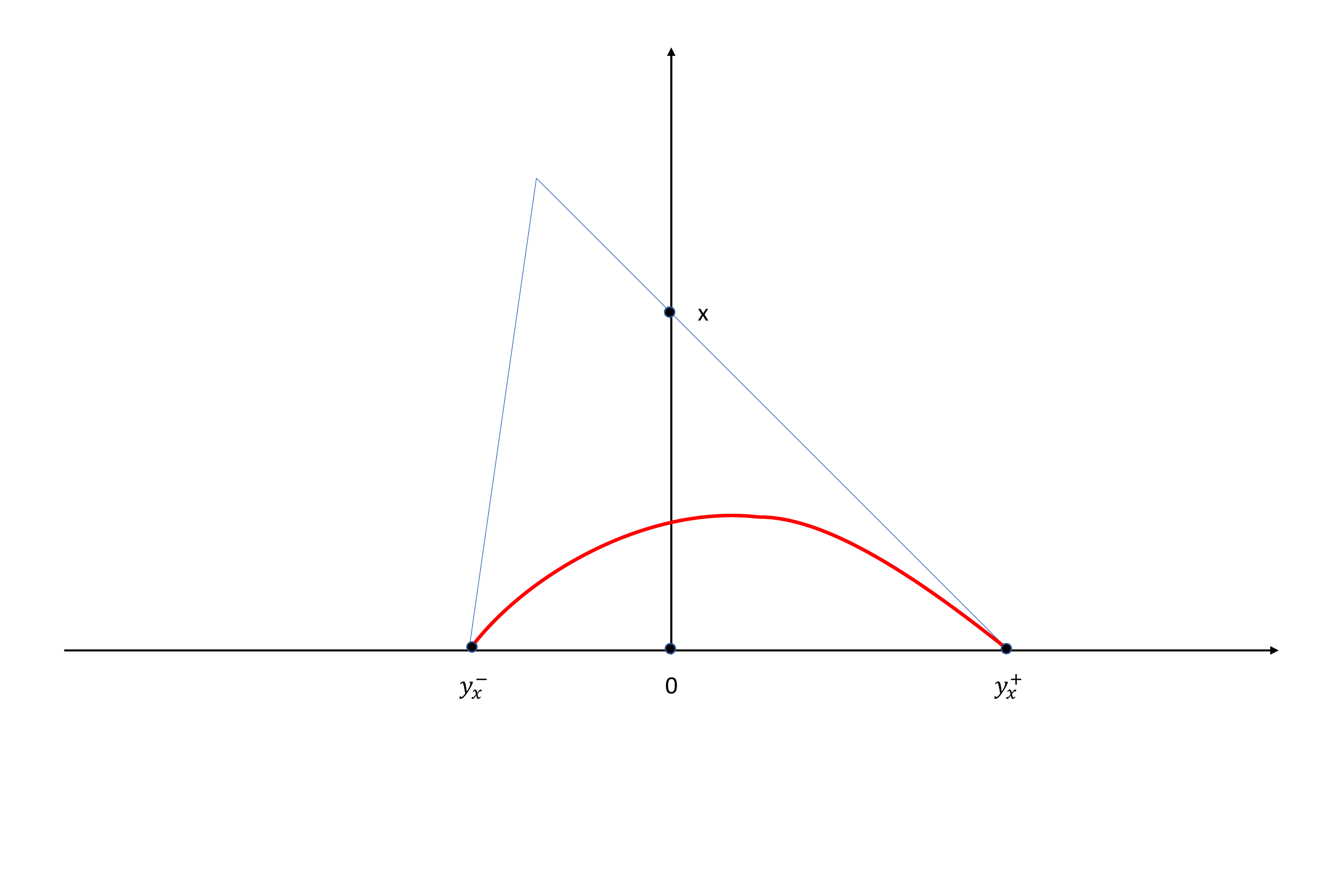}
\centering
\label{F-5-1}
\end{figure}
\begin{definition}
\label{D05}
In the Renewal EFT (EFF), an edge that starts before zero and ends after zero is called \textbf{flying over zero}.
\end{definition}
The following propositions give results about the $S$-set of the connected Renewal Bridge Graph.
\begin{proposition}
\label{Pr7}
The S-set has an a.s. infinite cardinality. 
\end{proposition}
\begin{proof}
Assume the cardinality of the $S$-set is a.s finite. Then, in the corresponding Renewal EFT, there are finitely many edges flying over zero.  That is, the set \\ 
\begin{equation}
F:=\{ t\in \mathbb{N};\text{\footnotesize{ there is a flying edge over zero in the Renewal EFT with the end vertex $t$}}\},
\end{equation}
\newline \normalsize{is a.s. finite. So in the Renewal EFT, all the vertices before vertex $0$ are the descendants of one of the vertices in the $F$ set. So at least one of the vertices in the $F$ set has infinitely many descendants, which is impossible because all trees of  the Renewal EFF are $\mathcal{I/I}$.}
\end{proof}
\begin {proposition}
\label{Pr8}
Almost surely, every vertex in the $S$-set has finitely many *-descendants in the Renewal Bridge Graph.
\end{proposition}
\begin{proof}
Let $(0,y)$ be a vertex in the S-set. The number of *-descendants of $(0,y)$ is equal to the number of the descendants of vertex $t^{+}_{y}$ in the Renewal EFF. Each connected component of the Renewal EFF is $\mathcal{I/I}$ . So this number is a.s. finite. 
\end{proof}
\begin{remark}
Note that the \textbf{finite descendant property} is not just for the vertices in the $S$-set and holds for the whole Bridge Graph. 
\end{remark}
\section{Properties of the Bridge Graph of a General Null Recurrent Markov Chain}
\label{S5}
This section extends most of the results on the Renewal Bridge Graph to the Bridge Graph of a general aperiodic, null recurrent Markov Chain. Many properties of the Renewal Bridge Graph are concluded from the coupling defined in \eqref{Co}. An object similar to the Renewal EFT will be used for the general Bridge Graph. 
\begin{definition}  Consider an aperiodic and null recurrent Markov Chain $\{X_{n}\}_{n \in \mathbb{N}}$, and its associated Bridge Graph ${B}_{X}$, with reference vertex, $s^{*}$. In this setup the time axis is $\{(t,s^{*});t\in\mathbb{Z})$. In  ${B}_{X}$, consider the paths starting from a point in the time axis for example $(t,s^{*})$, then look at the times when these paths get back to this axis again. Let the random variable  $\mathcal{T}_{t} $ denote the time that it takes for the path starting from vertex $(t,s^{*})$ in the Bridge Graph to return to the time axis for the first time. Define $G^{\mathcal{T}}=G^{\mathcal{T}}(V, E)$ be the random graph, where $V$, the set of the vertices, is the whole $\mathbb{Z}$, and where $E$, the set of the edges, is defined as follows: each point $t$ has a single outgoing edge, $e_{t}$, with length $\mathcal{T}_{t}$. This random graph is called the \textbf{Recurrence Time EFF} of ${B}_{X}$. In the connected case, it is referred to as the Recurrence Time EFT. 
\end{definition}
\begin{remark}  Note that in the Recurrence Time EFF:
\begin{enumerate} [label=\roman*]
\item Since the Markov Chain $\{X_{n}\}_{n \in \mathbb{N}}$ is null recurrent for all $t$, then $\mathbb{E}(\mathcal{T}_{t} )=\infty$.
\item The random variables $\{\mathcal{T}_{t}\}$ are identically distributed. However, since the paths that start from two different vertices $(t_{1},s^{*})$ and $(t_{2},s^{*})$ in the Bridge Graph may meet each other before returning to $s^{*}$, the random variables $\{\mathcal{T}_{t}\}$ are not independent in general, even in the totally independent case defined in Remark \ref{RT}.
\end{enumerate}
\end{remark}
\subsection{Properties of the Recurrence Time EFF and the Null Recurrent Bridge Graph} \hfill \\
 The Recurrence Time EFF has almost the same properties as the Renewal EFF. 
\begin{proposition}
\label{Pr9} 
The Recurrence Time EFF is a unimodular network. Also, it is in the $\mathcal{I/I}$ class of the foil classification theorem for unimodular networks, when it is connected. 
\end{proposition}
\begin{proof}
Since the Recurrence Time EFF  is constructed by a stationary marking of $\mathbb{Z}$, it is a unimodular random graph rooted at $0$ (see \cite{Aldous}). For the $\mathcal{I/I}$ structure property in the connected case, the same proof as for Proposition \ref{Pr2} for the connected case works here. 
\end{proof}
\begin{remark}
The Recurrence Time EFT, $G^{\mathcal{T}}$, is defined as a function of the Bridge Graph.  In the Bridge Graph, $B_{X}$, the mass transport principle holds for those vertices that belong to $G^{\mathcal{T}}$. This means that the triple $(B_{x}, A, 0)$ is locally unimodular in the sense of \cite{Tom}, with $A=\{(x,y);y = s^{*}\}$.
\end{remark}

So far, only the properties of Recurrence time EFFs were discussed. Their implications on  general null recurrent Bridge Graphs are now considered. 
\begin{proposition} 
\label{New3} Every bi-infinite path, $\{\beta_{t}\}_{t\in \mathbb{Z}} $, in the null recurrent Bridge Graph, $B_{X}$, corresponds to a bi-infinite path in its associated Recurrence time EFF.
 \end{proposition}
\begin{proof}
The proof of this proposition is almost the same as Proposition \ref{Pr3-4} in the Renewal case. It is first shown that there is no bi-infinite path in the null recurrent Bridge Graph that intersects the time axis for the first time at time $T$, where $T$ is an arbitrary time in the time axis.  Define the same set $D$, as in the renewal case: \\

$D=$\{$T-t; $ the path that starts from $(T-t,s^{*})$ in the null recurrent Bridge Graph goes back to the time axis for the first time at time $T$\}. \\

Note that the definition of $D$ gives that the vertices $(T-t,s^{*})$, in $B_{X}$, that belong to $D$, are, in its corresponding recurrence time EFF, the first generation descendants of $T$. The unimodularity of  the Recurrence Time EFF implies that a.s. all vertices have a finite degree. So the cardinality of $D$ is a.s. finite. \\
If there is a bi-infinite path, $\{\beta_{t}\}_{t \in \mathbb{Z}}$, in $B_{X}$, which is not bi-recurrent, then there is a time $T$ such that  $\{\beta_{t}\}_{t \leq T}$ does not intersect the time axis. Since $\{\beta_{t}\}_{t \leq T}$ is a path in the Bridge Graph, every vertex in it has a backtrack to the time axis. It means there exist infinitely many paths starting from a vertex $(t,s^{*}), t < T$, in the time axis and entering $\{\beta_{t}\}_{t \leq T}$, i.e., the starting vertices of these paths belong to the set $D$. This contradicts the fact that $D$ is a.s. finite for any arbitrary $T$. So such a path does not exist. 
\end{proof}
\begin{proof}[Proof of Proposition \ref{Pr13}]
Since the Renewal Bridge Graph is an example of null recurrent Bridge Graph, in the general case also, the null recurrent Bridge Graph can be either a tree or a forest. 
\end{proof}
\begin{proof}[Proof of Proposition \ref{Pr11}]
Proposition \ref{New3} gives that every bi-infinite path in the null recurrent Bridge Graph is bi-recurrent. So if there exists any bi-infinite path in the Bridge Graph, there is a bi-infinite path in the Recurrence Time EFT. Moreover, since the Recurrence Time EFT is $\mathcal{I}/\mathcal{I}$, there is no bi-infinite path in the null recurrent Bridge Graph. 
\end{proof}
\begin{proof}[Proof of Proposition \ref{Prnew}]
Proposition \ref{New3} gives that every bi-infinite path in $B_{X}$ corresponds to a bi-infinite path in the associated Recurrence Time EFF. \\
First, suppose that there is only one bi-infinite path in the Recurrence Time EFF. This bi-infinite path is a covariant subset, and since the Recurrence Time EFF is unimodular, with the same argument as in the proof of Proposition \ref{Pr2}, one can show that this is not possible. So either there is no bi-infinite path in the graph, or there are more than one bi-infinite path. In order to show that the further case is impossible, consider the following set:
\begin{equation}
M=\{(0,s) \in B_{X}; \text{$s$ is in a bi-infinite path}\}.
\end{equation}
Note that, the set $M$ is determined by the property of the Bridge Graph before time zero.  Suppose that $(0,s_{1})$ and $(0,s_{2})$ are two separate arbitrary elements in $M$. By assumption, with positive probability, the trajectories starting from these two vertices meet each other in the future, in the Bridge Graph. It means that, correspondingly, in the Recurrence Time EFF, two bi-infinite paths meet each other with positive probability. I.e., with positive probability, there is a connected component in the Recurrence Time EFF that has two bi-infinite paths, which is impossible due to the foil classification theorem of unimodular networks. So there is no bi-infinite path in the null recurrent Bridge Graph.   
\end{proof}
\begin{remark} Note that:~
\begin{enumerate}
\item From the proof of Proposition \ref{Prnew}, one can conclude that if the Recurrence Time EFF has this property that every two vertices in this graph meet each other with positive probability, then every connected component in this graph is $\mathcal{I}/\mathcal{I}$. 
\item The assumption of Proposition \ref{Prnew} does not put any condition on the MC itself, but it puts a condition on the coupling that exists between the driving sequence $\{\xi_{t}^{y}\}_{y \in S}$ when $t$ is fixed. In particular, in the case where $\{\xi_{t}^{y}\}_{y \in S}$ is totally independent, this condition is satisfied. This condition does not hold is Example \ref{MCLRW}, where the random variables $\{\xi_{t}^{y}\}_{y \in S}$, in this case, are maximally coupled for a fixed $t$. In this example, there are infinitely many bi-infinite paths in the Bridge Graph. 
\item  The Renewal Bridge Graph studied in Section \ref{S4} is an example where the driving sequence is totally independent. 
\end{enumerate}
\end{remark}
\begin{proposition}
\label{Pr14}
Foils in a connected component of the null recurrent Bridge Graph, $B_{X}$, are its intersections with vertical timelines. Thus the null recurrent Bridge Graph has infinitely many foils, and the order of foils is that of $\mathbb{Z}$. 
\end{proposition}
\begin{proof}
Consider the vertex shift $f$ in the Bridge graph $B_{X}$, where $f$ maps each vertex $(x,y)$ to its right adjacent vertex. For each $t \in \mathbb{Z}$, let
\begin{equation*}
S_{t}= \{(x,y)\in B_{X} ; x=t  \}.
\end{equation*}
Suppose that for some $t \in \mathbb{Z}$ , $(t, y_{1})$ and $(t, y_{2})$ are in $S_{t}$ and in the same connected component. Since they are in the same connected component, there is a vertex $(t_{0}, y_{0}) \in B_{X}$ such that the trajectories of these two vertices meet each other. By definition of the Bridge Graph, the number of steps that it takes for vertices $(t, y_{1})$ and $(t, y_{2})$ to reach $(t_{0}, y_{0})$ is equal to $t_{0}-t$, so these two vertices belong to the same foil. \\
For the converse, note that if two vertices $(t_{1}, y_{1})$ and $t_{2}, y_{2})$ are in the same foil, then the trajectories of these two vertices meet each other after the same number of steps. Since by definition of the Bridge Graph and the vertex shift $f$, at each step, the trajectories move exactly one unit forward in time, it follows that $t_{1}$ and $t_{2}$ are equal.  
\end{proof}
\begin{proof}[Proof of Proposition \ref{Pr12}]
Consider the Bridge Graph as a network, i.e., a graph with marks on its vertices and edges. Suppose that the Bridge Graph, $B_{X}$, is a unimodularizable network. In the general case, as shown in Proposition \ref{Pr14}, in each connected component, the intersections of the Bridge Graph with vertical lines are the foils, and the foils are a.s. infinite. The foils form a covariant vertex partition of the Bridge Graph in the sense of \cite{Baccelli2017}. Moreover, the time axis is a covariant subset of the Bridge Graph if vertices are marked by their coordinates with respect to the time and state axis. \\
Using the no infinite/finite inclusion lemma in  \cite{Baccelli2017} for unimodular networks, the intersection of a foil with the time axis should also be infinite, because each foil is almost surely infinite. However, in the Bridge Graph, this intersection has only one element. So the Bridge Graph is not a unimodularizable network in the sense of Definition \ref{D5-4}.
\end{proof}
\begin{remark}
When the Bridge Graph is considered as a graph without any marks with respect to the time axis and the state axis coordinates, it is not unimodularizable, in general as well. Indeed, this is the case when the time axis is distinguishable in the Bridge Graph as a covariant subset and hence the same proof as Proposition  \ref{Pr12} holds. For example, in the Renewal Bridge Graph, the set R is a distinguishable set in the whole graph, so this graph is not unimodularizable. However, considering the Bridge Graph without its marks can also lead to situations where it is unimodularizable. For instance, the maximally coupled Bridge Graph of the lazy random walk in Example \ref{MCLRW} is a unimodular Bridge Graph rooted at $(0,0)$. 
\end{remark}
The $S$-set of the general null recurrent Bridge Graph has the same properties as in the Renewal Bridge Graph.  
\begin{proposition}
\label{Pr15}
Under the assumptions of Proposition \ref{Pr11} (or Proposition \ref{Prnew}), the cardinality of the $S$-set of a null recurrent Bridge Graph is a.s. infinite. 
\end{proposition}
\begin{proof}
Proposition \ref{Pr9} states that the Recurrence time EFT of a null recurrent Bridge Graph is $\mathcal{I/I}$. Also, the proof of Proposition \ref{Prnew} shows that under its assumptions, the connected components of the Recurrence time EFT is $\mathcal{I/I}$. So the same proof as in the Renewal Bridge Graph in Proposition \ref{Pr7} holds here.
\end{proof}
\begin{proposition}
\label{Pr16}
Under the assumptions of Proposition \ref{Pr11} (or Proposition \ref{Prnew}), every vertex in the $S$-set of a null recurrent Bridge Graph has finitely many *-descendants a.s.
\end{proposition}
\begin{proof}
The proof of this property in the Renewal Bridge Graph was only based on its $\mathcal{I/I}$ nature. So the same proof holds here. 
\end{proof}
\section{$H^{T}$, and $H^{P}$ on the Null Recurrent Bridge Graph}
\label{Null}
This section considers the two dynamics, $H^{T}$ and $H^{P}$, defined in Section \ref{Main}  as dynamics on the measures on the state space of the null recurrent Bridge Graph, and studies their properties.\\
\subsection{Taboo Dynamics and Its Relation with the Invariant Measure of Null Recurrent MCs}The first dynamics is the Taboo Dynamics, $H^{T}$, defined in \eqref{004}. Theorem \ref{Propo1} states that this dynamics has at least one steady state.  
\begin{proof}[Proof of Theorem \ref{Propo1}]
Consider the Taboo PP as the initial state of the Taboo Dynamics. By the definition of this dynamics, at time one, there is mass one at state $s^{*}$.  Moreover, the mass of any arbitrary state $y \ne s^{*}$ is 
\begin{equation*}
{M}_{1}^{T}(y)= \sum _{x \in \mathcal{S} , x\ne s^{*}}{\tau}_{0}(x) \mathbbm{1}_{\{ h(x,\xi_{0}^{x})=y\}} + \mathbbm{1}_{\{ h(s^{*},\xi_{t}^{s^{*}})=y\}}.
\end{equation*}
Since for all $x\ne s^{*}$, $\tau_{0}(x)$ is the number of *-descendant of $x$ which are such that the first return to $s^{*}$ of the path starting from them takes place after time zero, ${M}_{1}^{T}(y)$ is also the number of *-descendants with the same property, that is $ \tau_{1}(y)={M}_{1}^{T}(y)$. It is clear from the Bridge Graph construction that $\tau_{1}\myeq \tau_{0} $. So the Taboo PP is a stationary state of the Taboo dynamics. 
\end{proof}
In the positive recurrent case, there is a known relation between the Bridge Graph and the stationary distribution of the Markov chain. In this last case, the Bridge Graph contains a unique bi-infinite path. So there is a point in the S-set with infinitely many descendants. It is proven in  \cite{Baccelli2018} that this point is a perfect sample of the stationary distribution of the Markov Chain. On the other hand, there is no bi-infinite path in the null recurrent Bridge Graph. All the points have finitely many descendants, and the approach of the positive recurrent case does not work. Instead, in the null recurrent case, the finite Taboo multiplicity defined in Definition \ref{D06} can be defined for the vertices on the $S$-set. Theorem \ref{T2} establishes a connection between this Taboo PP and the stationary measure of the null recurrent Markov Chain. \\
Before going to the proof of Theorem \ref{T2}, here is a classical lemma about computing stationary measure of the MC, using taboo probabilities. 
\begin{lemma}
\label{L02}
Let $X_{n}$ be an irreducible and recurrent Markov Chain. For fixed $i$ in the state space, let $\zeta$ be defined by
\begin{equation}
\label{5-7}
\zeta_{j}=\sum_{n=0}^{\infty} q_{ij}^{n}, \quad j=i, 
\end{equation}
where $q_{ij}^{n}$ is the probability for going from $i $ to $ j $, after $n$ steps, without visiting $i$. Then $\zeta$ is a positive invariant measure for the chain. This invariant measure is unique up to multiplication by a constant. 
\end{lemma}
\begin{proof}[Proof of Theorem \ref{T2}]
Let $G^{\mathcal{T}}$ be the Recurrence Time EFF associated with $B_{X}$, and $Q_{t}$ be the path that starts from $(t,s^{*})$ in $B_{X}$, where $s^{*}$ is the reference point of the Bridge Graph.\\
Consider an arbitrary state $y$ in the state space of the MC. Define $g[G^{\mathcal{T}}, t, t']=1$, when $Q_{t}$ passes through vertex $(t',y)$ before returning to $s^{*}$. The mass transport principle states that, for each $y$ in the state space $S$, the following equation holds:
\begin{equation}
\label{5-16}
\mathbb{E} [\sum _{t\in\mathbb{Z}}g[G^{\mathcal{T}}, t, 0]] = \mathbb{E} [\sum _{t\in\mathbb{Z}}g[G^{\mathcal{T}}, 0, t]].
\end{equation}
The right-hand side of \eqref{5-16} is the expectation of the number of times that path $Q_{0}$ intersects the state $y$ before returning to $s^{*}$, whereas the left-hand side is equal to the expectation of the mass that the Taboo PP puts at point $y$, at time $0$. Since the right-hand side of the equation does not dependent on the coupling of $\{ \xi_{t}^{y} \}_{y\in S}$ for a given $t$, the same is true for the left hand side. So the expectation of the Taboo point process at each point in the state space does not dependent on the coupling of $\{ \xi_{t}^{y} \}_{t\in \mathbb{Z},y\in S}$, for fixed $t$. Moreover, using Lemma \ref{L02}, the right-hand side of \eqref{5-16} is equal to the invariant measure of the MC, and this shows that Equation \eqref{5-80} holds for the mean measure of the Taboo PP. 
\end{proof}
\begin{remark}
Notice that, by Definition \ref{D06}, the point $s^{*}$ always belongs to the $S$-set. Moreover, the mass that the Taboo PP puts at this point is equal to $1$ at all times. 
\end{remark}

\subsubsection{On the Constructibility of the Taboo Point Process}
 \label{SS4}
Let vertex $(y,0)$ belongs to the $S$-set, the support of the Taboo PP. Then due to Proposition \ref{Pr16}, $\tau_{0}(y)$, is a.s. finite. In this sense, one can say that the Taboo PP is locally finitely constructible, i.e., the taboo multiplicity at each vertex of the $S$-set depends on a finite subtree of the Bridge Graph. \\
The important question here is whether the Taboo multiplicities are algorithmically constructible. A positive answer to this question is equivalent to the possibility of producing a perfect sample of these two point processes. In general, the answer to this question is unknown. However, there are some results for the special case, where the MC is stochatically monotone. These results are given in the  following subsection. 
\subsubsection{Stochastically Monotone Markov Chain}
\label{711}
Here it is shown that in the case where the transition probabilities of the Markov chain are stochastically monotone, the Taboo point process is algorithmically locally constructible. The following definition is borrowed from \cite{Monotone}.
\begin{definition}
Assume that the state space of the Markov chain $\{ X_{n}\}_{n \in \mathbb{Z}}$, $\mathcal{S}$, is endowed with a partial order denoted by $\leq$. The Markov chain is stochastically monotone if its transition probabilities  have the stochastic monotone property. i.e., the probability measures $(P(x,.); x \in \mathcal{S})$, on $\mathcal{S}$ are such that $P(x',.)\preceq P(x,.)$ whenever $x'\leq x$, whereby $\preceq$ means stochastically less than or equal to. 
\end{definition}
\begin{proposition}
\label{Pr180}
Assume $X=\{ X_{n}\}_{n \in \mathbb{Z}}$ is a stochastically monotone Markov Chain on $\mathcal{S}$, and $\mathcal{S}$ has the minimum element $s_{0}$. Moreover, suppose that the Bridge Graph $B_{X}$ is constructed with the reference point ${s^{*}=s_{0}}$. Then the TPP is algorithmically locally constructible. 
\end{proposition}
\begin{proof}
By a classical coupling argument, there exists a coupling $\xi_{t}(.)$ for the driving sequence $\{\xi_{t}^{x}\}$, at each time $t$, such that if $y\leq z$, then $\xi_{t}(y) \leq \xi_{t}(z)$ (See \cite{Prop} and \cite{Paolo}). \\
Denote the path that starts from $(t_{0},s^{*}), t \in \mathbb{Z}$, by $\mathcal{P}^{t_{0}}= \{\mathcal{P}^{t_{0}}_{t}, t\geq t_{0}\} $. Note that if $t_{0}<t_{1}$, then 
\begin{equation*}
\{\mathcal{P}^{t_{0}}_{t_{1}} \} \geq s_{0}= \{\mathcal{P}^{t_{1}}_{t_{1}}\},
\end{equation*}
and hence, with the monotone coupling argument, the path $\mathcal{P}^{t_{0}}$, at each time, remains larger than or equal to the path $\mathcal{P}^{t_{1}}$.\\
For constructing the Taboo PP, first start trajectories from time $-1$ and step by step add trajectories. Due to the monotone coupling and the later argument, each new trajectory that is added is larger than or coalesces with the pointwise supremum of the trajectories considered before. So, when a new trajectory is added to the Bridge Graph, the image of this trajectory might belong to one of the following scenarios: 
\begin{enumerate}[]
\item Adds no mass to the S-set, 
\item Adds mass to the last point added before,
\item Creates a new point in the $S$-set, which is greater than all the points that have been built in the $S$-set before.
\end{enumerate}
With this observation, once a new point appears in the $S$-set, the mass of the Taboo PP for the points that are less than or equal to this new point is fully determined. 
\end{proof}
\begin{remark}
Consider the random walk defined on the state space $\mathbb{N}$ with the following transition probabilities: for $n=0$, the walk stays at $0$, with probability $1/2$, and move to state $1$ with probability $1/2$. For state $n \in \mathbb{N}, n \ne 0$, the walk moves to each neighbor of $n$, i.e., $n+1$ and $n-1$ with probability $1/2$. This Markov chain is an example that satisfies the assumption of the MC in Proposition \ref{Pr18}. 
\end{remark}
With this algorithm described above, when a new path is added to the Bridge graph from time $t$, it might add a new mass to the $S$-set or not. So computing the expectation of the time until a new mass is added to the $S$-set gives information about the time it takes to construct the taboo multiplicities locally. This expectation is computed in the following corollary. 
\begin{lemma}
\label{L03}
Consider the assumptions of Proposition \ref{Pr180} and the algorithm that is presented in its proof. If $\mathcal{P}^{t_{1}}$ is a path that adds mass to the $S$-set, then for all $t< t_{1}$, $\mathcal{P}^{t}$ does not intersect the time axis (the state $s^{*}$ of the state space) between times $t_{1}$ and $0$. 
\end{lemma}
\begin{proof}
Consider the path $\mathcal{P}^{t_{0}}$, with $t_{0}< t_{1}$. For all $t > t_{1}$, $\mathcal{P}_{t}^{t_{0}} \geq \mathcal{P}_{t}^{t_{1}}$. Since the path $\mathcal{P}^ {t_{1}}$ adds mass to the $S$-set; it does not intersect the time axis up to time zero, so the same holds for the path $\mathcal{P}^{t_{0}}$.
\end{proof}
\begin{corollary}
Under the assumptions of Lemma \ref{L03},  Suppose that $t_{1}$ and $t_{1}-T$ are two successive times such that the paths $\mathcal{P}^{t_{1}}$, and $\mathcal{P}^{t_{1}-T}$ add mass to the $S$-set, Then 
\begin{equation}
\mathbb{E}[T]=\infty.
\end{equation}
\end{corollary}
\begin{proof}
Since $t_{1}-T < t_{1}$, Lemma \ref{L03} gives that for all $t< t_{1}-T$, the path $\mathcal{P}^{t}$ does not intersect the time axis at time $t_{1}$. So vertex $(t_{1},s^{*})$ does not have any descendants before time $t_{1}-T$ in $B_{X}$. Consequently, in the Recurrence Time EFT, the vertex $t_{1}$ does not have descendants before vertex $t_{1}-T$. So,  
\begin{equation}
\label{5-19}
\mathbb{E}[T]>\mathbb{E}[\text{Number of descendants of $t_{1}$ in the Recurrence Time EFT}].
\end{equation}
Note that the backward construction of the Bridge Graph, as mentioned before, starts from time zero and explores the graph step by step in the past. When it is constructed up to time $t_{1}$, the left-hand side of $t_{1}$, in the Bridge Graph, is not explored yet. So the distribution of the Bridge Graph before time $t_{1}$ is the same as the original distribution of the Bridge Graph. The original distribution in the Recurrence Time EFT, is such that the expectation of the number of the $*$-descendants of any vertex is infinite. So the right-hand side of Equation \eqref{5-19} is infinite, and thus $\mathbb{E}[T]=\infty$. 
\end{proof}
\subsection{Potential Dynamics} The second dynamics defined on the null recurrent Bridge Graph is the Potential Dynamics. Theorem \ref{Pr18} states that this dynamics also has at least one steady state.
\begin{proof}[Proof of Theorem \ref{Pr18}]
Consider the Potential PP as the initial state of the Potential Dynamics, where the dynamics is associated to a null recurrent MC. Let $y\ne s^{*}$ be a state in the state space $\mathcal{S}$. The mass that the Potential Dynamics puts at $y$ at time $1$, is equal to 
\begin{equation}
{M}_{1}^{P}(y)=
    \sum _{x \in \mathcal{S}}{\pi}_{0}(x) \mathbbm{1}_{\{ h(x,\xi_{0}^{x})=y\}}.
\end{equation}
That is, it is obtained by adding all the masses that enter the state $y$, via the Bridge Graph's edges, from time $0$. Since for each $x\in \mathcal{S}$, ${\pi}_{0}(x)$ is equal to the number of its $*$-descendants, ${M}_{1}^{P}(y)$ is equal to the number of $*$-descendants of vertex $(1,y)$ in the Bridge Graph, which is  the potential multiplicity of this vertex. The same argument shows that the result holds for $y=s^{*}$, with this difference that  $s^{*}$ itself should be counted once. So the Potential PP is a steady-state of the Potential Dynamics. 
\end{proof}
Theorem \ref{T3}  shows the connection between the steady state of the Potential Dynamics and the null recurrent MC associated with it. This connection is related to the potential matrix  $R$ of the MC with entries $R_{xy}$, where $R_{xy}$  is the expected number of visiting of the state $y$ given that the initial state of the MC is the state $x$. For a recurrent MC, the entries of the potential matrix are all equal to infinity. 
\begin{proof}[Proof of Theorem \ref{T3}]
Let $G^{\mathcal{T}}$ be the Recurrence Time EFT (EFF) associated with $B_{X}$, and $Q_{t}$ the path that starts from $(t,s^{*})$ in $B_{X}$, where $s^{*}$ is the reference point of the Bridge Graph.\\
Consider an arbitrary state $y$ in the state space of the MC. Define $g[G^{\mathcal{T}}, t, t']=1$, when $Q_{t}$ passes through the vertex $(t',y)$. Using the mass transport principle, for each $y$ in the state space $S$ the following equation holds:
\begin{equation}
\label{5-16}
\mathbb{E} [\sum _{t\in\mathbb{Z}}g[G^{\mathcal{T}}, t, 0]] = \mathbb{E} [\sum _{t\in\mathbb{Z}}g[G^{\mathcal{T}}, 0, t]].
\end{equation}
The right-hand side of the \eqref{5-16} is the expectation of the number of times that path $Q_{0}$ intersects the state $y$, which its expectation is equal to $R_{s^{*}y}$. On the other hand, the left-hand side of \eqref{5-16} is equal to the expectation of the mass that the Potential PP puts at point $y$, at time $0$. Since $R_{s^{*}y}$ is related to a recurrent MC, the mean measure of the Potential PP at each point is infinity. 
\end{proof}
Consider the null recurrent Bridge Graph. Proposition \ref{Pr16} states that in this case, for each $(0,y)$ in the $S$-set, $\pi_{0}(y)$ is a.s. finite. So one  can consider the algorithmic constructibility of the Potential PP as the Taboo PP. The same result as Proposition \ref{Pr180} is valid for the Potential PP.  
\begin{proposition}
\label{Pr190}
Assume $X=\{ X_{n}\}_{n \in \mathbb{Z}}$ is a stochastically monotone Markov Chain on $\mathcal{S}$, and $\mathcal{S}$ has the minimum element $s_{0}$. Moreover, suppose that the Bridge Graph $B_{X}$ is constructed with the reference point ${s^{*}=s_{0}}$. Then the Potential PP is algorithmically locally constructible. 
\end{proposition}
\begin{proof}
Consider the backward construction algorithm for the Bridge Graph with the monotone coupling introduced in the proof of Proposition  \ref{Pr180}. The same argument as in the proof of Proposition \ref{Pr180} shows that, once a new point appears in the $S$-set, the mass of thePotential PP for the points that are less than or equal to this new point is fully determined. So the claim is proved. 
\end{proof}
\begin{remark}
Let $X$ satisfy the assumptions of Proposition \ref{Pr190}, and $B_{X}$ be its associated Bridge Graph constructed by the monotonically coupled driven sequence.  Although the same step-by-step backward construction algorithm for the Bridge Graph gives the potential multiplicities of the point in the $S$-set, one can consider a faster algorithm for constructing the potential multiplicities. \\
For showing this, let $(t_{0},s^{*})$ and $(t_{1},s^{*})$, where $t_{0}<t_{1}$, be two vertices in the time axis that add mass to the same vertex $(0,y)$ in the $S$-set when the potential multiplicity is considered. Then all the vertices $(t_{i},s^{*})$, where $t_{0}<t_{i}<t_{1}$, add mass to the potential multiplicity of $(0,y)$. So, for finding the Potential multiplicity of $(0,y)$, it is sufficient to find the first time and the last time that add mass to $(0,y)$. So instead of constructing the Bridge Graph step by step, one can use the \emph{exponential search} algorithm (see \cite{EX}) to find the last time that adds mass to vertex $(0,y)$. Then use \emph{binary search} for finding the first time. If $T$ is the position of the search time, then exponential search takes $O(logT)$ time to find $T$. So this new algorithm for finding the Potential multiplicity is faster than the step-by-step construction.
\end{remark}
\section{Perfect Sampling of the Taboo and Potential PPs in the Critical Single Server Queue}
\label{sec:loy}
\subsection{Loynes' Theory}
This section is focused on the application of the results of the previous sections to the $GI/GI/1$ queue and its associated workload Markov Chain \cite{loynes_1962}. The service times $\{\varsigma_{n}; n \in \mathbb{Z} \}$ are assumed to be i.i.d.  with finite and nonzero mean $\mathbb{E}[\varsigma]$. The inter-arrival times are i.i.d. and denoted by $\{\upsilon_{n}; n \in \mathbb{Z} \}$, with finite and nonzero mean $\mathbb{E}[\upsilon]$.  That is, $\upsilon_{n}=T_{n}-T_{n-1}$, where $T_{n}$ is the arrival time of the $n$-th customer. Let $W_{n}=W(T_{n}-)$ denote the workload just before time $T_{n}$, which is the amount of service remaining to be done by the server at that time. 
This workload process satisfies the equation
\begin{equation}
\label{Workload}
W_{n+1}=(W_{n}+\varsigma_{n}- \upsilon_{n})^{+}, \quad \forall n \in \mathbb{Z},
\end{equation}
where $(a)^{+}=\max(a,0)$. Because of the i.i.d. assumptions, (\ref{Workload}) defines an $\mathbb N$-valued Markov Chain. To avoid degenerate cases, it is assumed below that the
variance of $\varsigma-\upsilon$ is non zero and that $\{\varsigma_n\}_n$
and $\{\upsilon_n\}_n$ are independent, although these assumptions are not essential. The \textit{traffic intensity} is $\rho=\frac{\mathbb{E}[\varsigma]}{\mathbb{E}[\upsilon]}$. It is well known that when $\rho<1$, this Markov Chain is positive recurrent and when $\rho>1$, it is transient. In the critical case, when $\rho=1$, it is null recurrent. \\ 
In the case where $\rho<1$, Loynes' theory allows one to define a perfect sample from the stationary distribution \cite{loynes_1962}. 
This section extends this theory to the perfect sampling of the Taboo PP using the algorithm that is provided in Subsection \ref{711}, which applies since this Markov Chain is stochastically monotone.

The Loynes variable at time $n\ge 0$, $L_n$, is the value of the workload at time 0
when starting the queue empty at time $-n$, and when coupling the service and
inter-arrival times as in the CFTP algorithm, namely
$$ L_n= \left(\max_{k=1,\ldots,n}\sum_{l=-n}^{-1}(\varsigma_l-\upsilon_l)\right)^+.$$
The sequence $\{L_n\}_n$ is non-decreasing.
\subsection{Interpretation and Perfect Sampling of the Taboo PP}
It is easy to see that $M_0^T$, the Taboo PP of this Markov Chain, is simple, and that its support is $\{L_n\}_{n\ge 0}$.
Indeed, adding customers $-n$ in the past leads to a new atom in this PP if and only if
this addition creates a busy period that starts at the arrival of customer $-n$ and lasts until time 0.
The Taboo PP at time 0 is hence equal to 
\begin{equation}
\label{QTA}
 M^T_0= \delta_0+ \sum_{n\ge 1} \delta_{L_n} 1\{L_n>L_{n-1}\}.
 \end{equation}
It captures the {\em joint structure of the workload strict records} in this Loynes type (or CFTP) construction.

In the positive recurrent case, the Loynes sequence converges to an a.s.
finite limit $L_\infty$; the Taboo PP is then a.s. finite and the supremum value of its
support is the perfect sample $L_\infty$ of the steady state workload.
In the null recurrent case, it has an infinite support but is locally finite.

Figure \ref{TabooSample} provides perfect samples of the random measure $M^{T}_0$ restricted to bounded intervals 
for different inter-arrival and service time distributions.
The samples are perfect because the chain is stochastically monotone.

A corollary of the result on the first moment measure of the Taboo PP is that,
in the recurrent case, the invariant measure of this
Markov Chain admits the representation
\begin{equation}
\label{NEWTABOO2}
 \sigma(i)=\mathbb{E} M^T_0(i)=\sum_{n\ge 1} \mathbb{P}[L_n=i,L_n>L_{n-1}], \quad \forall i>0
\end{equation}
and $\sigma(0)=1$. The terms in this expression are reminiscent of ladder epochs and heights, but differ from those in that they bear on the backward rather than the forward workload sequence.

\begin{figure*}
        \centering
        \begin{subfigure}[b]{0.475\textwidth}
            \centering
            \includegraphics[width=\textwidth]{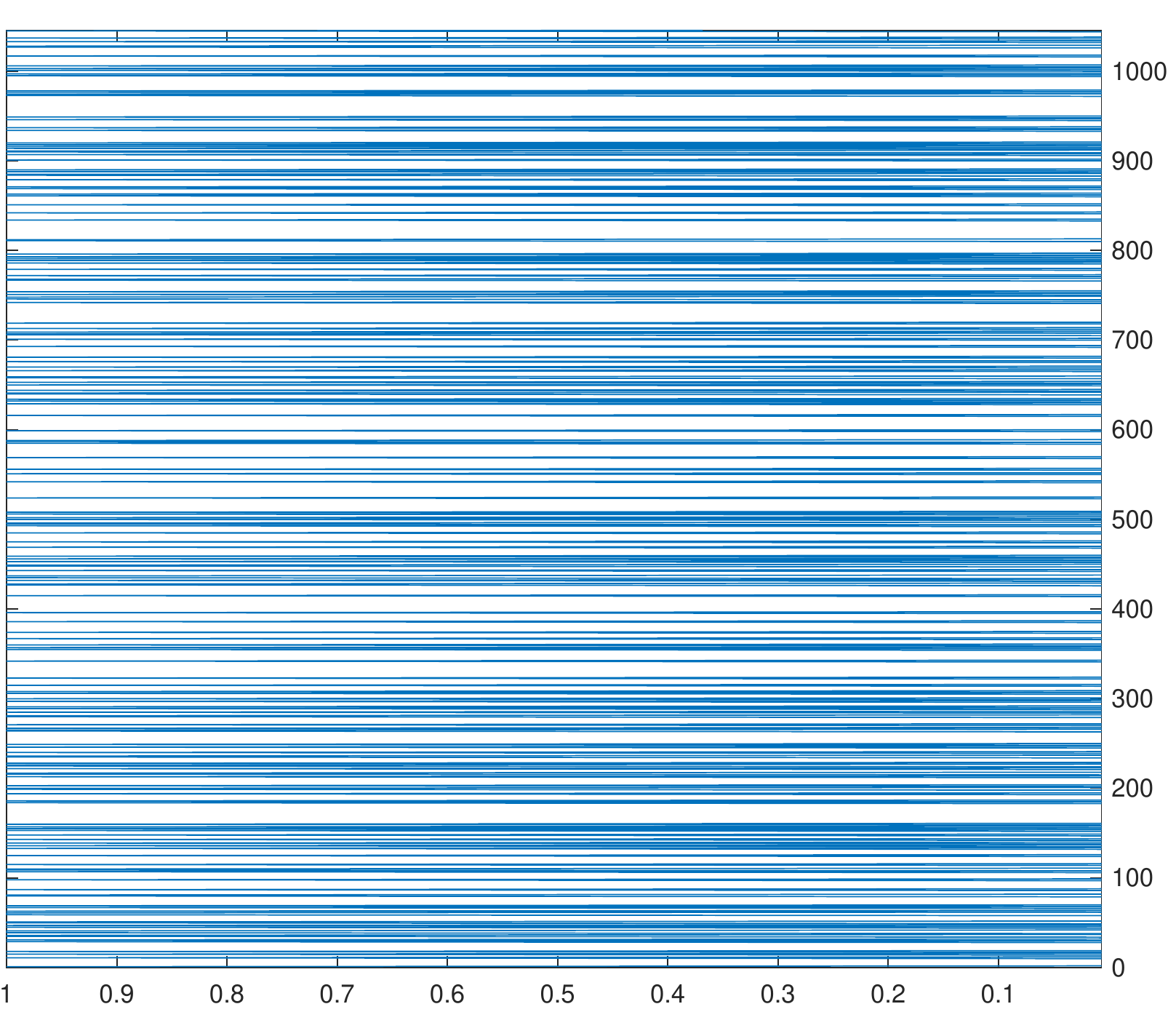}
            \caption[] {{\small $\varsigma, \upsilon \sim Geo(0.2)$, and $0 \leq i \leq 1000$}}    
            \label{GE01}
        \end{subfigure}
        \hfill
        \begin{subfigure}[b]{0.475\textwidth}  
            \centering 
            \includegraphics[width=\textwidth]{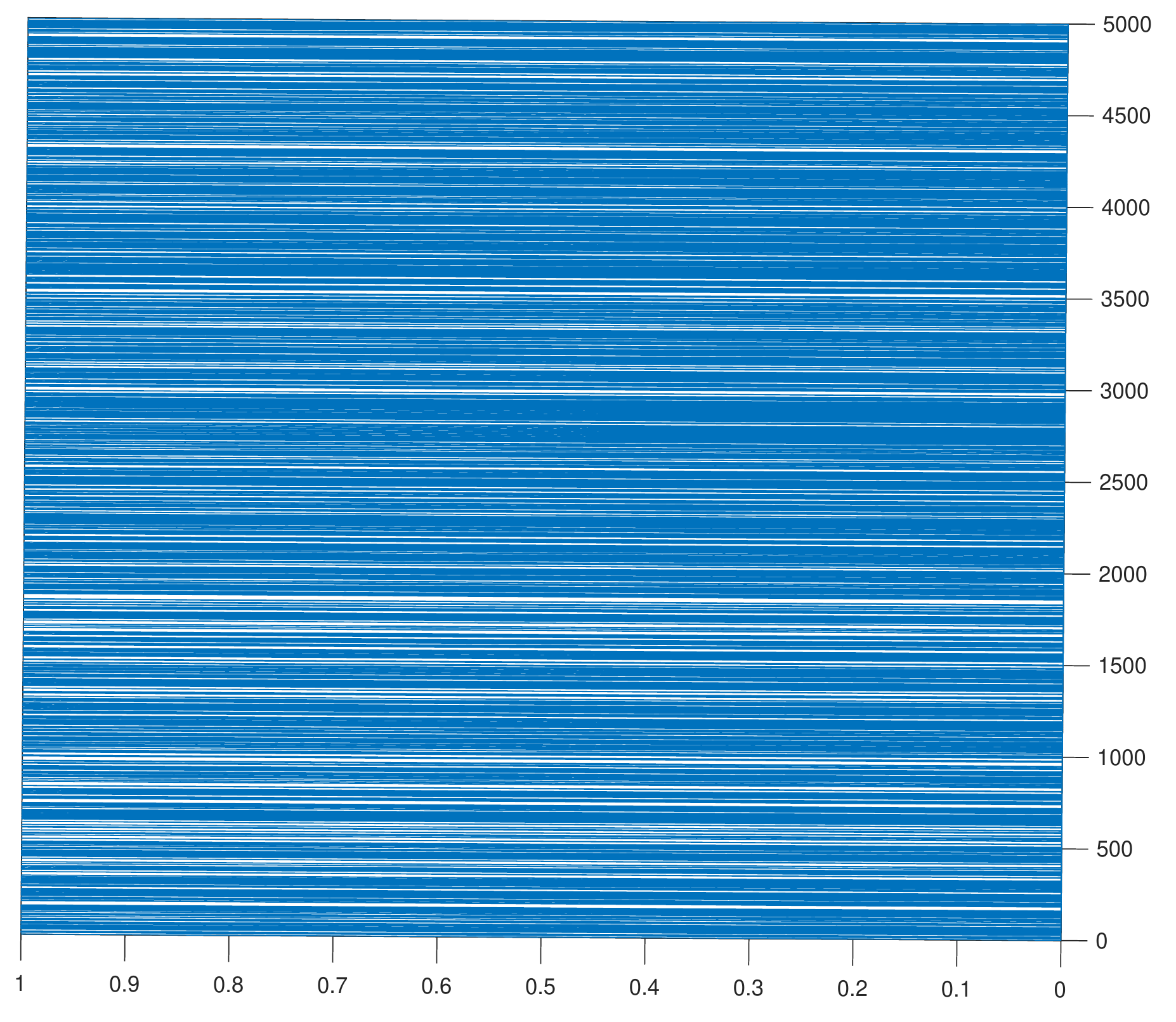}
            \caption[]%
            {{\small $\varsigma, \upsilon \sim Geo(0.2)$, and $0 \leq i \leq 5000$}}    
            \label{fig:mean and std of net24}
        \end{subfigure}
        \vskip\baselineskip
        \begin{subfigure}[b]{0.475\textwidth}   
            \centering 
            \includegraphics[width=\textwidth]{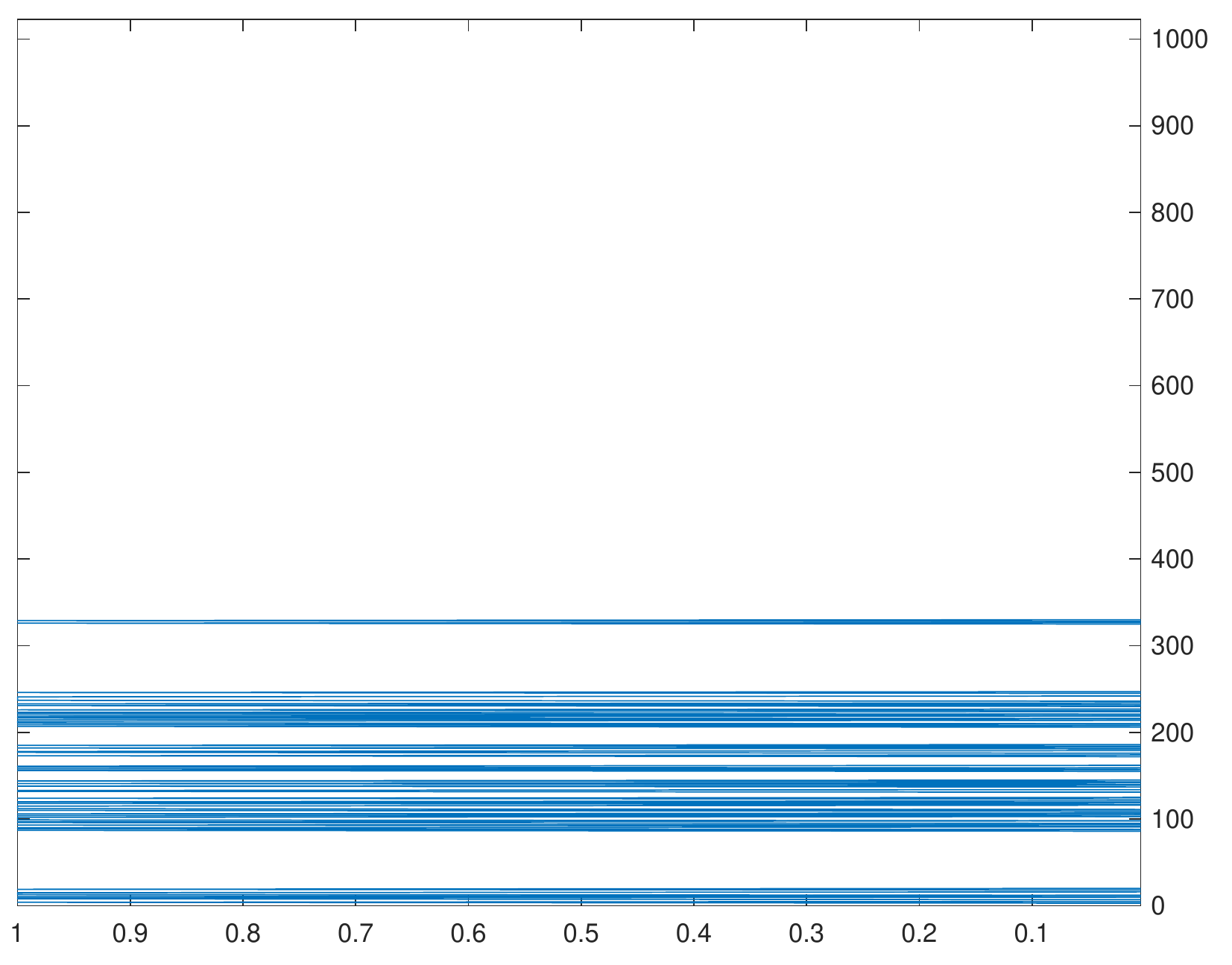}
            \caption[]%
            {{\small $\varsigma, \upsilon \sim Zeta(2.5)$, and $0 \leq i \leq 1000$}}    
            \label{fig:mean and std of net44}
        \end{subfigure}
        \hfill
        \begin{subfigure}[b]{0.475\textwidth}   
            \centering 
            \includegraphics[width=\textwidth]{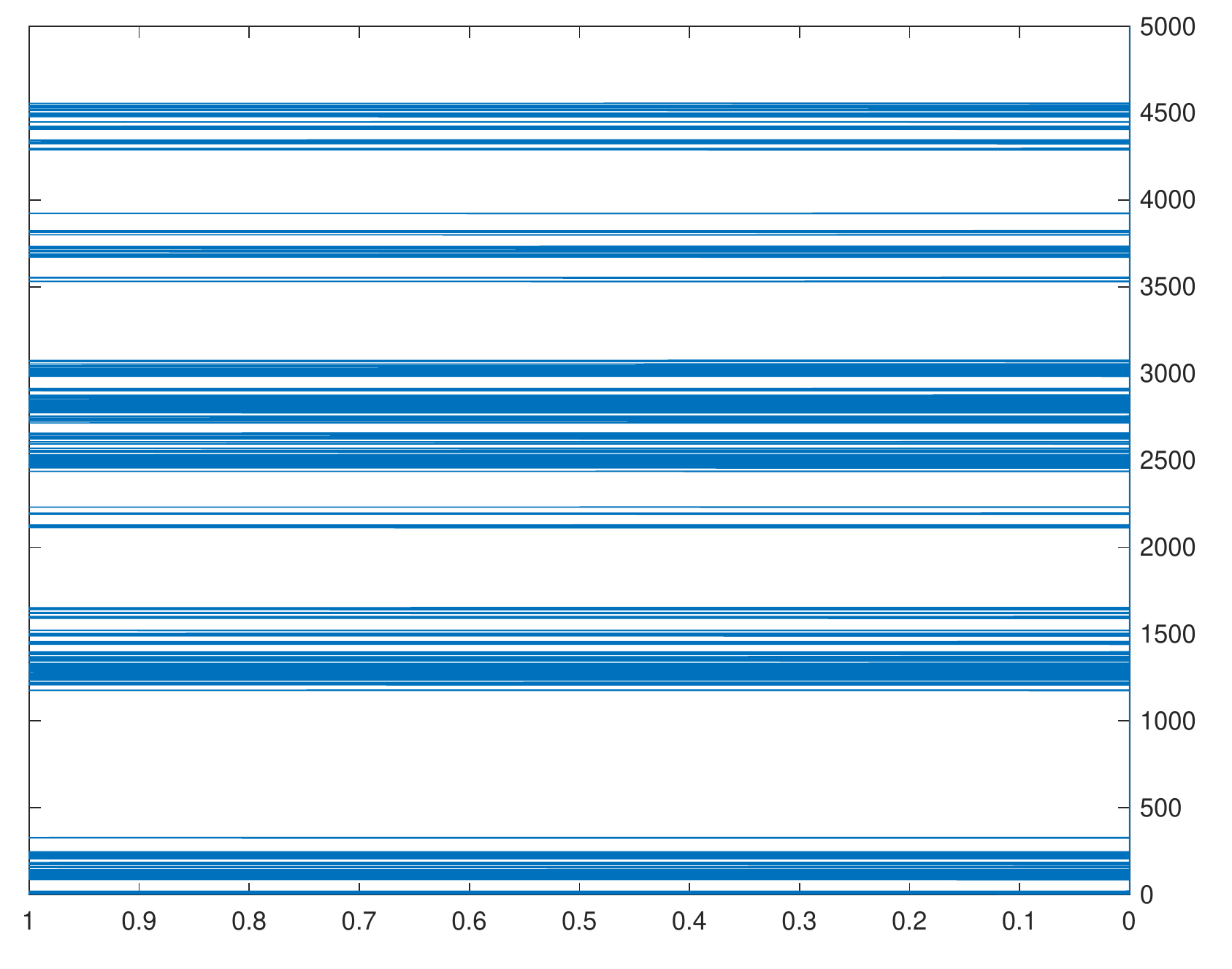}
            \caption[]%
            {{\small $\varsigma, \upsilon \sim Zeta(2.5)$, and $0 \leq i \leq 5000$}}    
            \label{fig:mean and std of net44}
        \end{subfigure}
        \caption[ Perfect samples of $M^{T}$ in the null recurrent case for two different distributions of inter-arrival and service times, for two different intervals of the state space. ]
        {\small Perfect samples of $M^{T}$ in the null recurrent case for two different distributions of inter-arrival and service times, for two different intervals of the state space. } 
        \label{TabooSample}
    \end{figure*} 
    
 One can analyze various properties of the Taboo PP using these perfect samples. As Equation \eqref{NEWTABOO2} shows, 
 the first moment measure of $M^{T}_0$, i.e., its intensity measure, is the invariant measure of the workload Markov Chain. 
 Moment measures of order 2 of $M^{T}_0$ give some information about the interaction between the points. 
 Using an idea similar to that of the Ripley K-function in point process theory, one can detect clustering or 
 inhibition in $M^{T}_0$ by comparing this function to 1, see \cite{Stoyan}. For this, consider the following local second-order moment based function:
 \begin{equation*}
 K_{i}(r)=\frac{\mathbb{E}[ M^{T}_0(i-r,i+r)|M^{T}_0(i)=1]}{\mathbb{E}[ M^{T}_0(i-r,i+r)]}.
 \end{equation*}
This function can be estimated using perfect samples of $M^{T}_0$. If the points were distributed independently, 
for all $i$, we would have $K_{i}(r)=1$; this value is used as a benchmark. 
If $K_{i}(r)>1$, there is clustering at point $i$, whereas if $K_{i}(r)<1$, there is inhibition at point $i$ for radius $r$. 
Figure \ref{TabooSampleK} shows estimates of $K_{i}(r)$ based on a large collection of perfect samples in some examples of critical queues. As the figures show, there is no general conclusion about clustering or inhibition of $M^{T}_0$ in this monotone case. The analyzed examples suggest that when inter-arrival and service time variances are finite, there is inhibition for small $r$, and the value of $K(r)$ tends to $1$ for large $r$ (Figures \ref{GE01} and \ref{GE02}). In contrast, when inter-arrival and service time variances are infinite, there is clustering for small $r$ (Figures \ref{Zeta01} and \ref{Zeta02}). 
\begin{figure*}
        \centering
        \begin{subfigure}[b]{0.475\textwidth}
            \centering
            \includegraphics[width=\textwidth]{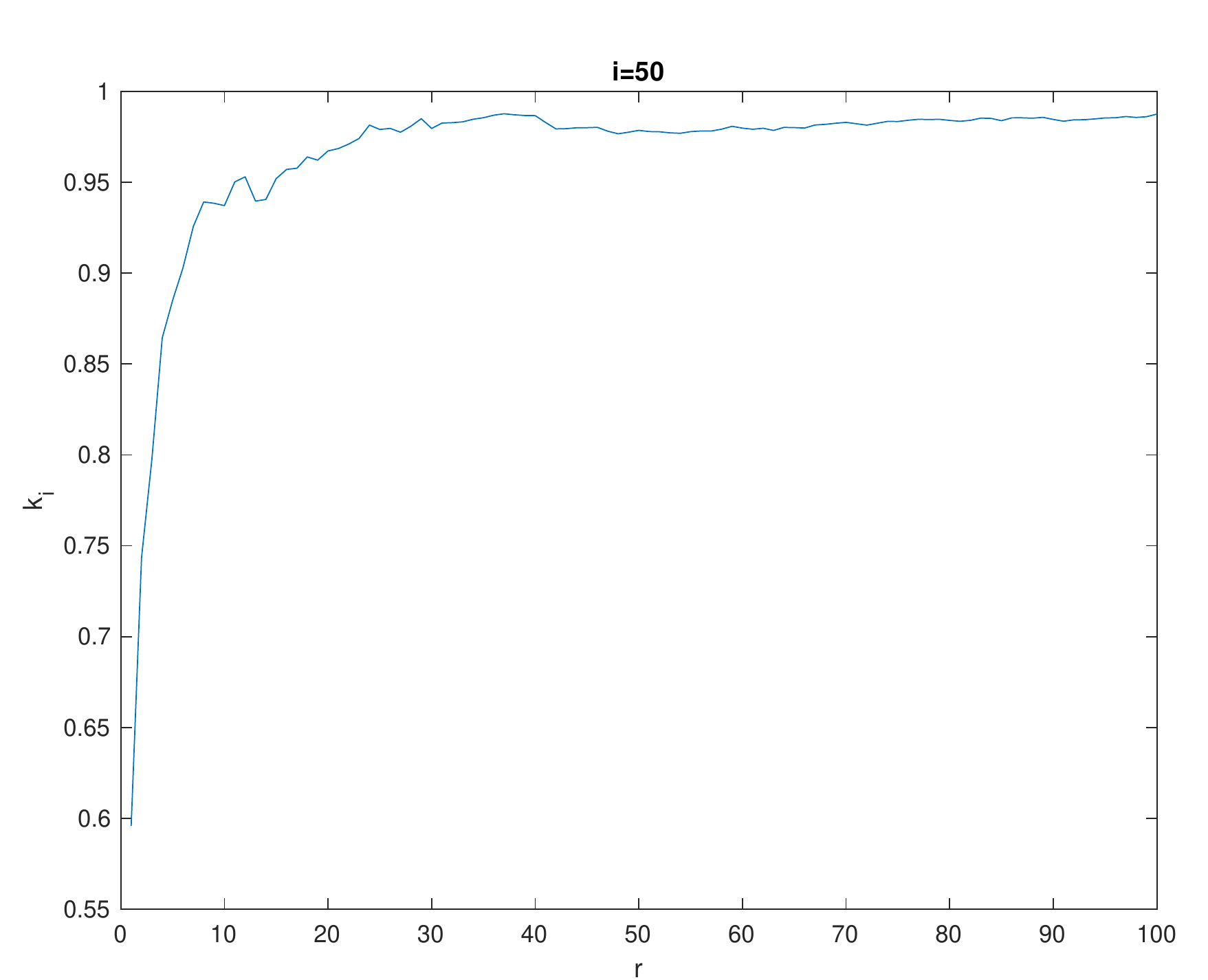}
            \caption[] {{\small $\varsigma, \upsilon \sim Geo(0.2)$}}    
            \label{GE01}
        \end{subfigure}
        \hfill
        \begin{subfigure}[b]{0.475\textwidth}  
            \centering 
            \includegraphics[width=\textwidth]{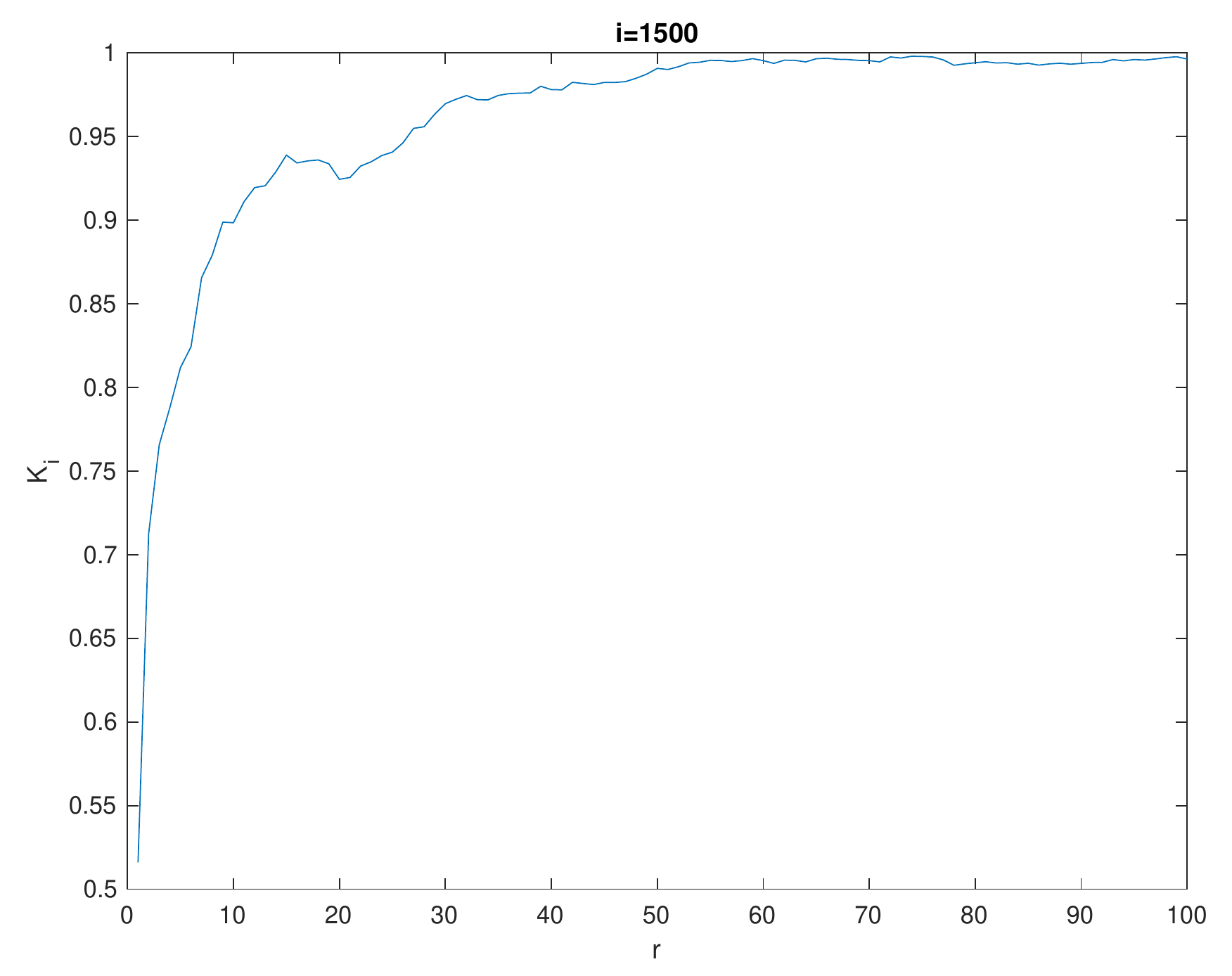}
            \caption[]%
            {{\small $\varsigma, \upsilon \sim Poi(25)$}}    
            \label{GE02}
        \end{subfigure}
        \vskip\baselineskip
        \begin{subfigure}[b]{0.475\textwidth}   
            \centering 
            \includegraphics[width=\textwidth]{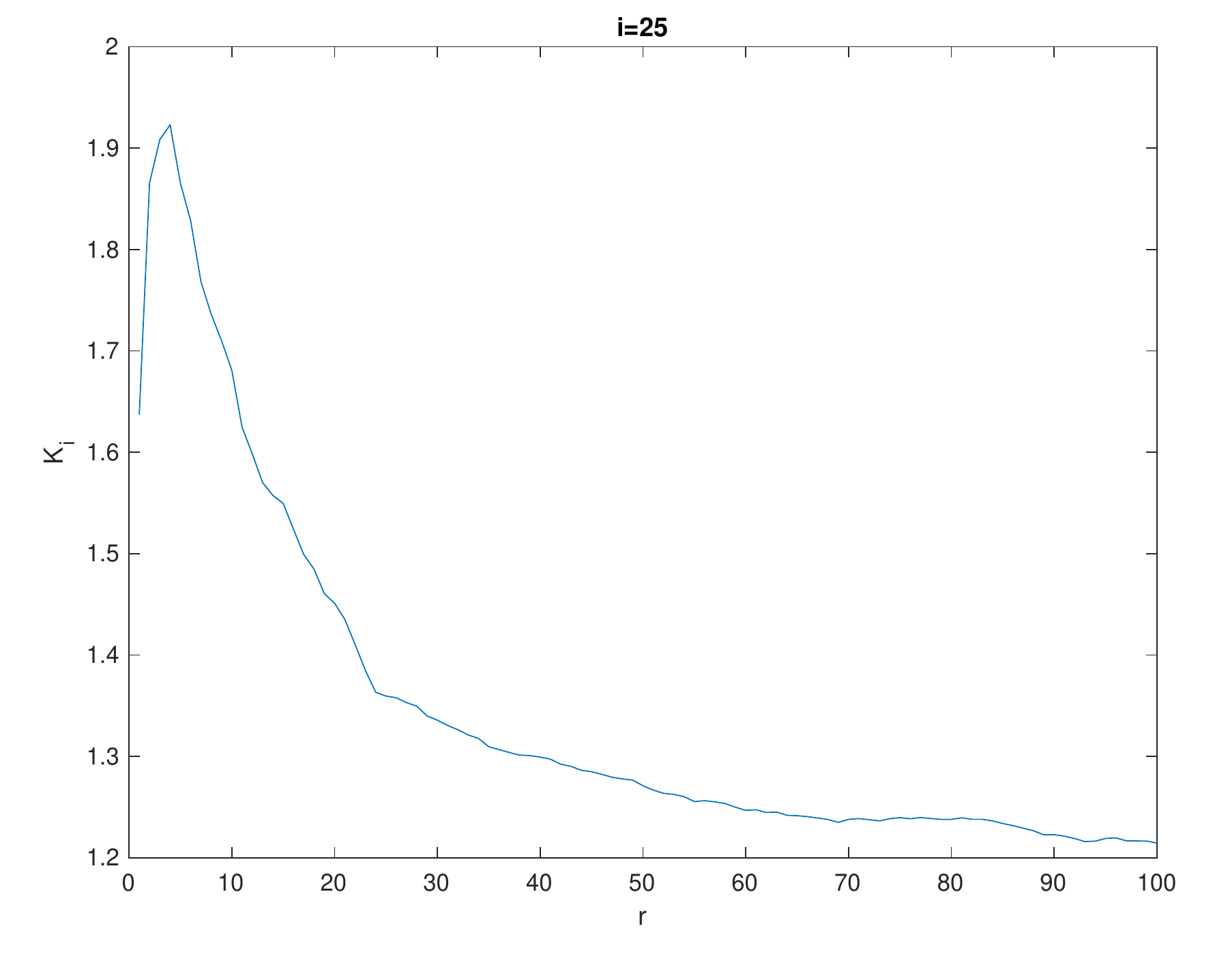}
            \caption[]%
            {{\small $\varsigma, \upsilon \sim Zeta(2.5)$}}    
            \label{Zeta01}
        \end{subfigure}
        \hfill
        \begin{subfigure}[b]{0.475\textwidth}   
            \centering 
            \includegraphics[width=\textwidth]{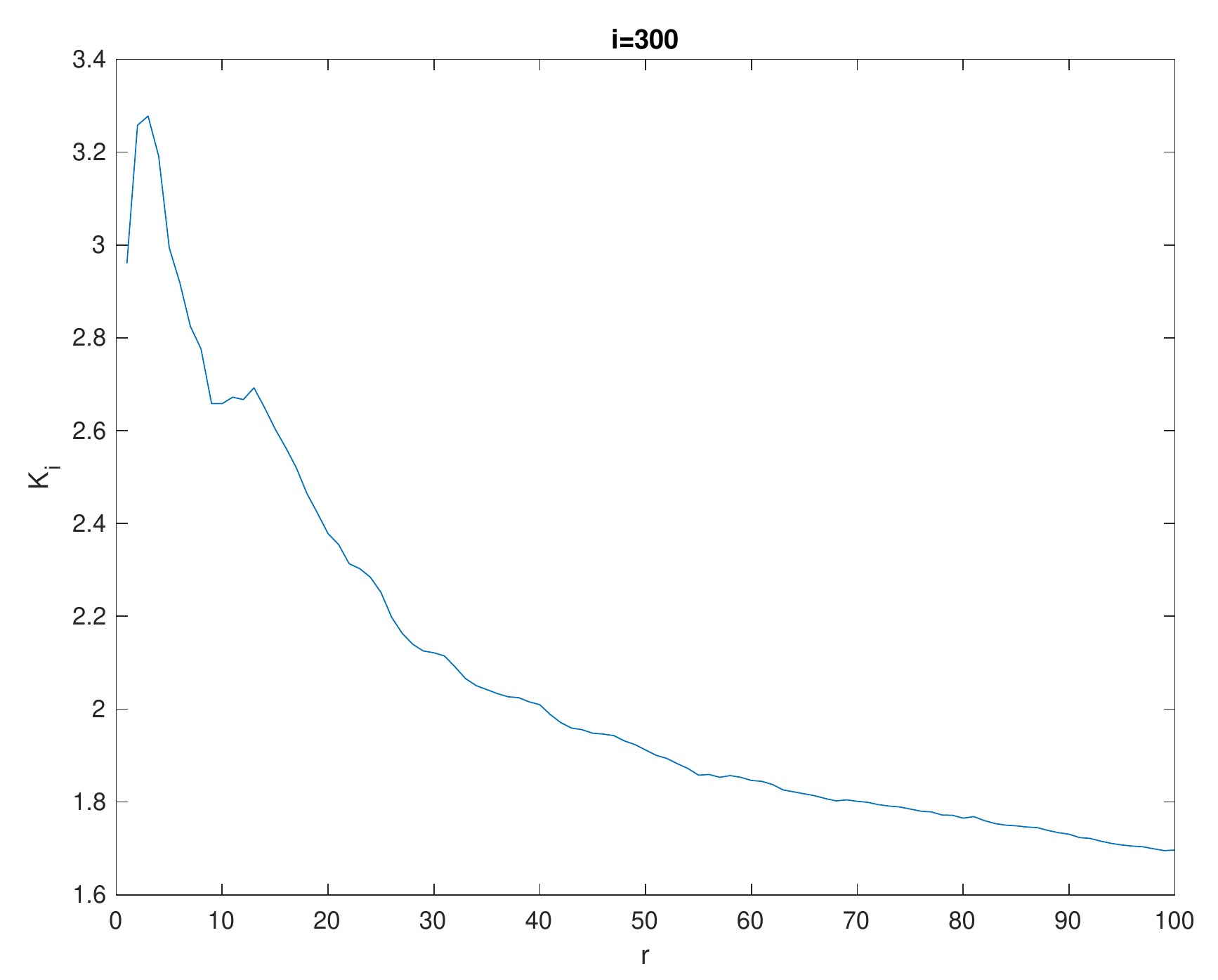}
            \caption[]%
            {{\small $\varsigma, \upsilon \sim Zeta(2.5)$}}
            \label{Zeta02}
        \end{subfigure}
        \caption[ Estimation of $K_{i}(r)$ with 1000 samples of $M^{T}_0$, for two different values of $i$, and $0<r\leq 100$.]
        {\small Estimation of $K_{i}(r)$ with 1000 samples of $M^{T}_0$, for two different values of $i$, and $0<r\leq 100$. } 
        \label{TabooSampleK}
    \end{figure*} 
    
\subsection{Interpretation and Perfect Sampling of the Potential PP}
The potential PP has the same support as the Taboo PP, but
different multiplicities. It is easy to see that if $i$ is an atom of the Taboo PP, the multiplicity
of atom $i$ in the Potential P.P. is the number of epochs
that separate in the backward construction
the inclusion of atom $i$ in the Taboo PP, due to an increase
of the Loynes variable, from its last increase (and atom inclusion).
That is
\begin{equation}
\label{QPO}
M_0^P=\sum_{n\ge 0} \delta_{L_n} 1_{L_n>\max_{0\le k\le n-1} L_k}
\left( \sum_{k\le n} 1_{L_n=L_k} \right).
\end{equation}

It follows from our general results that, in the null recurrent case, this random measure is a.s.
locally finite, though with an infinite first moment measure, whereas it is not locally finite
in the positive recurrent case. In other words, this Potential PP gives the {\em joint time-space structure of the records in Loynes' construction},
with the support of this PP describing the spatial organization of the backward records, as for the Taboo case, and
the multiplicities describing their time separation. 

Figure \ref{POTENTIAL} gives a perfect sample of an instance of Potential PP at different scales. This point process inherits the complex ``correlation'' structure of the Taboo PP through their common support. The fact that it 
has an infinite intensity measure means that, in addition, all its multiplicities are heavy tailed. These last two properties together with the CFTP space-time interpretation discussed above contribute
making this Potential point process a fascinating object.

The interpretation of the Potential PP survives in the positive recurrent case, with the caveat that it is not locally finite. 
In fact, in this particular case, all atoms except the largest one
have a finite multiplicity, with the same time separation interpretation as above. However, the largest
atom, namely $L_\infty$ has an infinite multiplicity as it belongs to the bi-infinite path.

It is easy to check that the expressions \eqref{QTA} and \eqref{QPO} hold beyond the queuing context, provided the Markov Chain 
satisfies the stochastic monotonicity assumption.

\begin{figure*}
        \centering
        \begin{subfigure}[b]{0.475\textwidth}
            \centering
            \includegraphics[width=\textwidth]{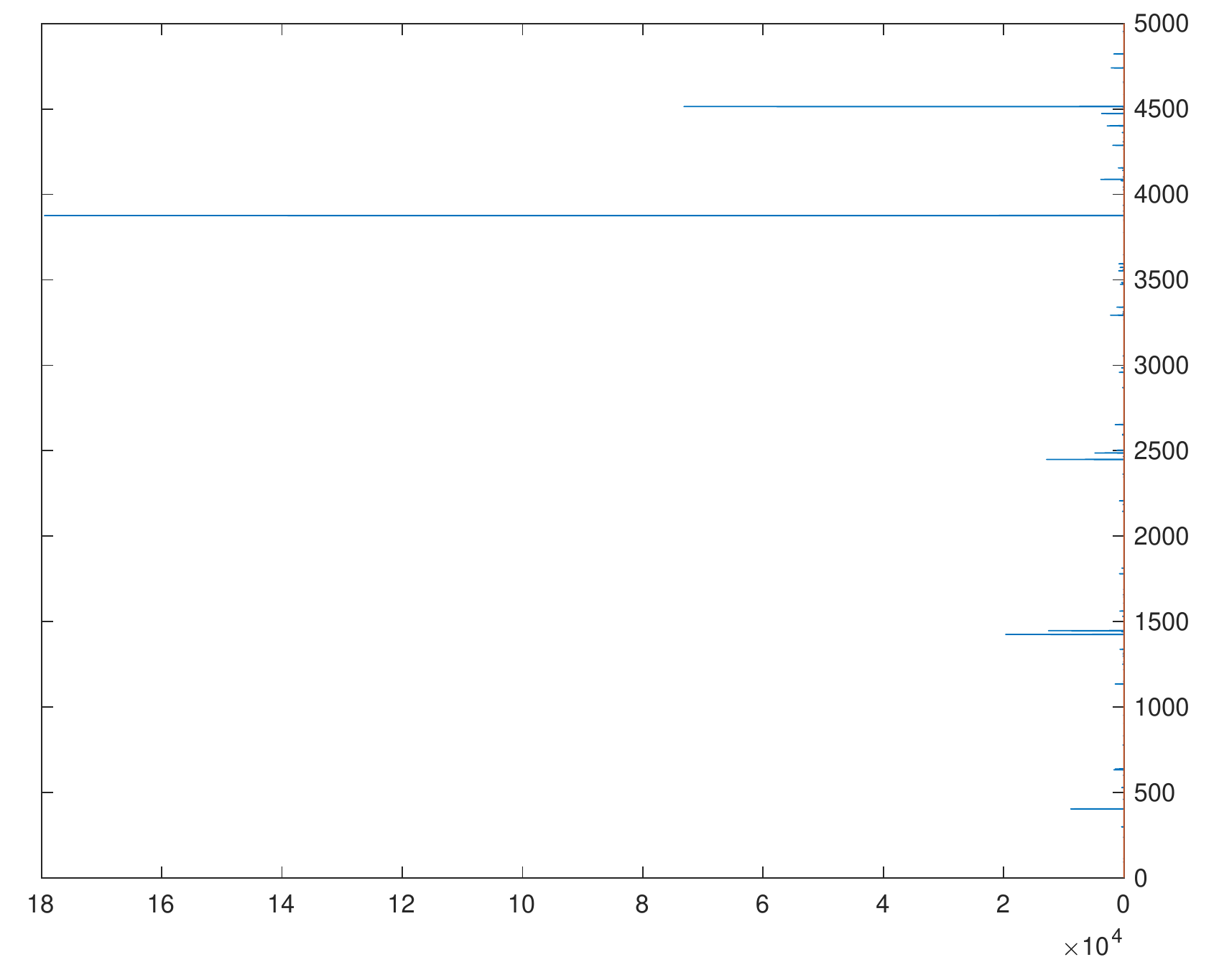}
            \caption[] %
           
        \end{subfigure}
        \hfill
        \begin{subfigure}[b]{0.475\textwidth}  
            \centering 
            \includegraphics[width=\textwidth]{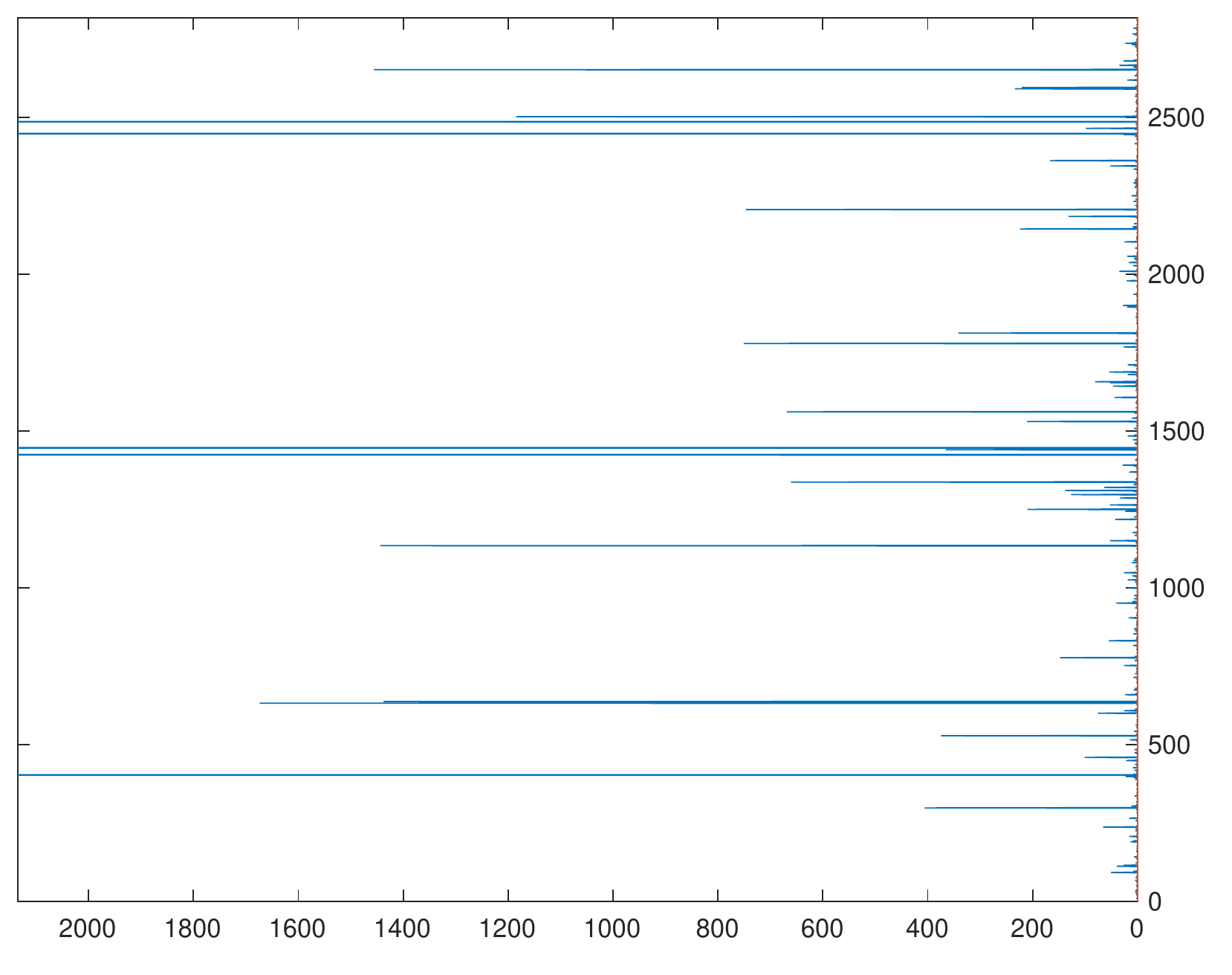}
            \caption[] %
           
        \end{subfigure}
        \hfill
        \begin{subfigure}[b]{0.475\textwidth}  
            \centering 
            \includegraphics[width=\textwidth]{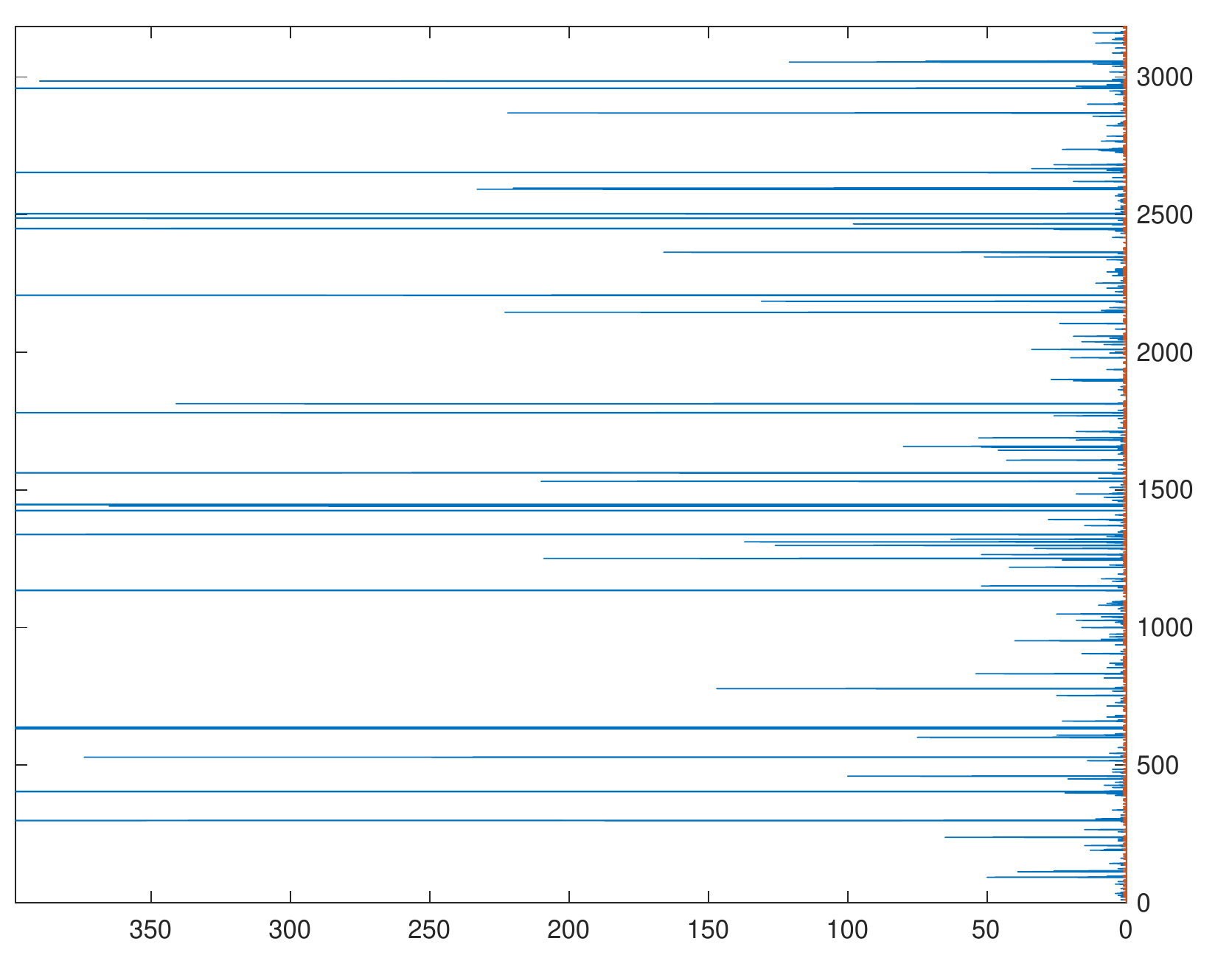}
            \caption[]%
          
        \end{subfigure}
        \hfill
        \begin{subfigure}[b]{0.475\textwidth}  
            \centering 
            \includegraphics[width=\textwidth]{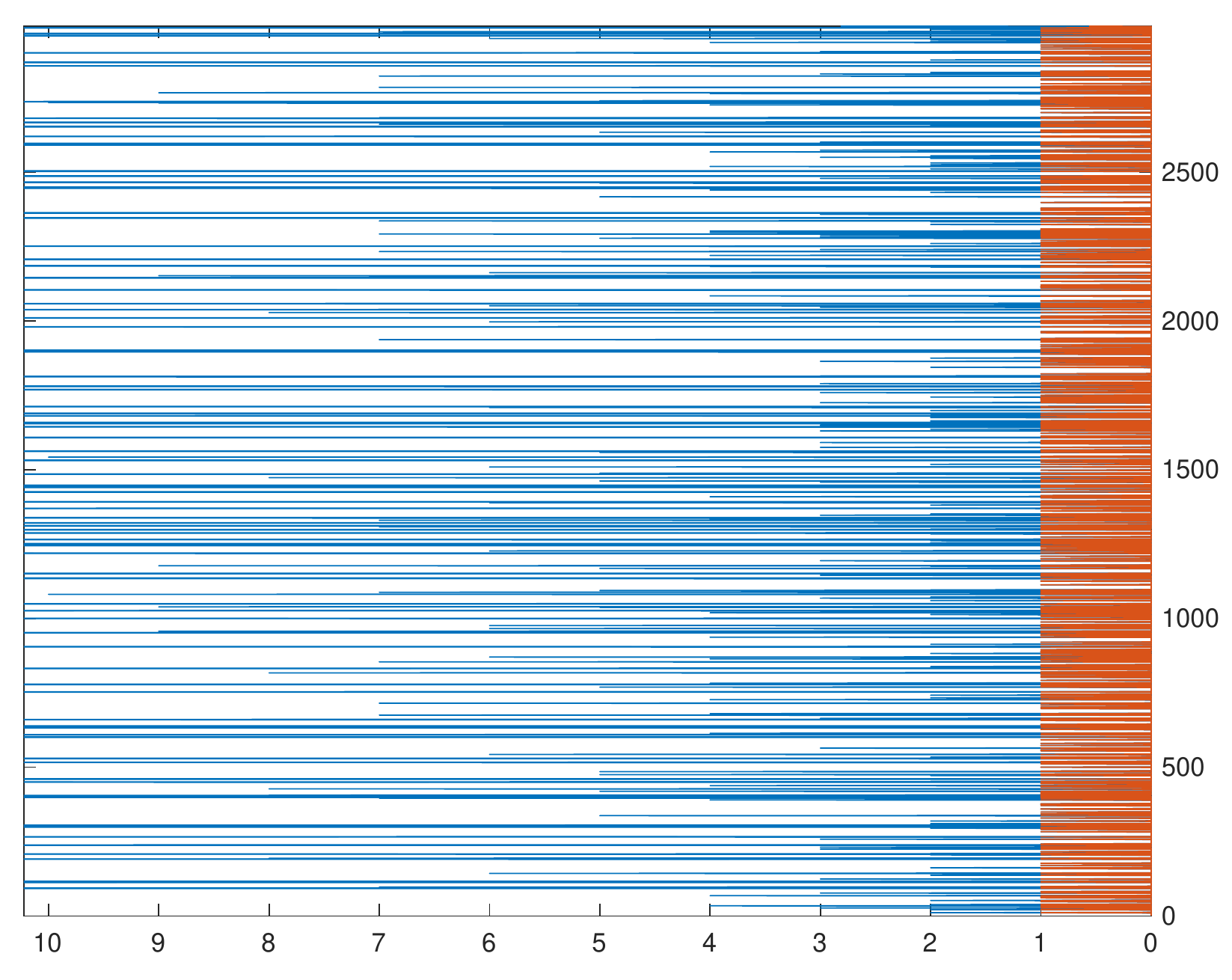}
            \caption[]%
           
            \label{}
            
        \end{subfigure}
        \vskip\baselineskip
        
        \caption[ Perfect samples of $M^{P}_0$ for Poisson-distributed inter-arrival and service time, for different scales]
        {\small Perfect samples of $M^{P}_0$ for Poisson-distributed inter-arrival and service time, for different scales, the red segments represent the Taboo PP.} 
        \label{POTENTIAL}
    \end{figure*} 

\section{ Taboo PP on the Positive Recurrent Bridge Graph}
\label{Possitive}
\subsection{Positive Recurrent and Null Recurrent Bridge Graph}
The positive recurrent Bridge Graph is studied in \cite{Baccelli2017}. It is known that the positive recurrent Bridge Graph is a.s. connected when the driving sequence is totally independent, it is unimodularizable, and in the sense of the foil classification theorem for unimodular networks, it is $I/F$. Moreover, the $I/F$ property of the positive recurrent Bridge Graph gives that this graph contains a unique bi-recurrent path $\{\beta_{t}\}_{t\in\mathbb{Z}}$. A positive recurrent MC, $X$, has a unique stationary distribution, and the importance of the bi-infinite path $\{\beta_{t}\}_{t\in\mathbb{Z}}$, in its associated Bridge Graph, $B_{X}$, is that the intersection of this path with the zero timeline gives a perfect sample of the stationary distribution of $X$.
On the other hand, in the null recurrent Bridge Graph, it is shown in the current work that this graph is not connected in general; it can be a tree or forest. Moreover, it is not unimodularizable in general. However, it ``contains'' a unimodular random network  (the Recurrence Time EFT/EFF) which allows one to prove some properties of the null recurrent Bridge Graph. In contrast with the positive recurrent case, in the null recurrent Bridge Graph, there is no bi-infinite path when it is connected (or under the assumption of Proposition \ref{Prnew}). However, there is an analog of the perfect sample of the positive recurrent case, namely the Taboo PP.
\begin{remark}
Of course, there are other ways for constructing a point process that its intensity is equal to the invariant measure at each point of the state space. For example, consider a path of the MC, $X$, starting from an arbitrary state $s^{*} \in \mathcal{S}$. For each $s\in \mathcal{S}$, consider the number of times this path meets $s$ before going back to $s^{*}$. The expectation of this number for each $s\in\mathcal{S}$ is equal to $\zeta_{s}$, defined in \eqref{5-7}, which is equal to the stationary measure of $s$. So in each realization of the Markov Chain starting from $s^{*}$, this number can be considered. The advantage of Taboo PP as a technique of sampling is its local constructibility (at least in the monotone case), as mentioned in Subsection \ref{SS4}. 
\end{remark}
One can consider the Taboo PP and its properties in the positive recurrent Bridge Graph. A question that arises here is that of the relationship between the Taboo PP of a positive recurrent MC and the classical perfect sample of its stationary distribution. 
\begin{proposition}
\label{Pos1}
Consider a positive recurrent MC, $X$, and its associated Bridge Graph, $B_{X}$. 
Then the $S$-set of  $B_{X}$ is a.s. finite. Moreover, the Taboo multiplicity of every vertex in the $S$-set is a.s. finite. 
\end{proposition}
\begin{proof}
In the positive recurrent case, the Bridge Graph is an $I/F$ unimodular network \cite{Baccelli2018}. Also, the foils in the Bridge graph are its intersection with the vertical timelines. So the S-set, which is the $0$- foil, is a.s. finite. \\
For proving the second part, note that since $B_{X}$ is $I/F$, every vertex $y \in B_{X}$, which is not on the bi-infinite path, has a.s. finitely many descendants, specially finitely many *-descendants. To complete the proof, note that, although there are infinitely many $*$-descendants on the bi-infinite path, only one of them does not return to $s^{*} $ before time zero. Hence the bi-infinite path adds exactly mass one to the $S$-set. 
\end{proof}
\subsection{Relation between the Taboo PP and Classical Perfect Sampling}
The positive recurrent MC, X, has a unique stationary distribution, $\sigma$. One can sample from this stationary distribution with the Coupling from the Past algorithm (See \cite{Prop}). The Taboo PP also gives a samples from the stationary distribution, in the sense that the mean measure of the Taboo PP is the stationary distribution at each point. 
The relation between these two samplings is discussed in the next proposition: \\
\begin{proposition}
Let $\tau$ be the Taboo PP of the positive recurrent MC, $X$. Suppose that $T$ is a sample of $\tau$. If one biases $T$ with the number of the points (considering the multiplicities) that belong to it, and chooses a random point from $T$ and denote it by $Y$, then $Y$ has the stationary distribution of $X$. 
\end{proposition}
\begin{proof}
Let $\mathcal{\tau}$ be the set of all possible outcomes of the Taboo PP. For each $T \in \mathcal{T}$, let $P_{T}$ be the probability that event $T$ occurs, i.e., $\mathbb{P}(\tau=T)$. Then for each $y \in \mathcal{S}$,
\begin{align*}
\label{5-15}
\mathbb{P}(Y=y)&=\sum_{T \in \mathcal{T}} P_{T} \times \frac{m(T)}{\sum_{T' \in  \mathcal{T}} m(T')P_{T'}} \times \frac{T(\{y\})}{m(T)} \nonumber \\
&= \frac{1}{{\sum_{T' \in  \mathcal{T}} m(T')P_{T'}}} \sum_{T \in \mathcal{T}} P_{T} \times T(\{y\}) 
=\frac{\mathbb{E}(\tau(y))}{\mathbb{E}(\tau(\mathcal{S}))},
\end{align*}
where $m(T)$ is the sum of the multiplicity of the vertices in $T$, and $T(\{y\})$ is the multiplicity of $y$ in $T$. 
Since this probability is proportional to $\mathbb{E}(\tau(y))$, and the stationary measure in $y$, $\sigma (y)$, is also proportional to $\mathbb{E}(\tau(y))$,  
$\mathbb{P}(Y=y)= \sigma (y)$.
\end{proof}
Given a realization of the Taboo PP, a natural question is whether is it possible to get a perfect sample of $\sigma$ from this realization in the classical sense ? \\
Here is an algorithm for this, under the extra assumption that $M$ exists such that
$M>m(T) \quad \forall T \in \mathcal{T}$.\\

\noindent
{\bf {Algorithm}}: 
\begin{enumerate}
\item Generate a sample $T$ from $\tau$. 
\item Choose a point $Y$ randomly from $T$ with probability proportional to its multiplicity in $T$.
\item Accept the point $Y$ with probability $\frac{m(T)}{M}$.
\item If the point $Y$ is rejected, back to step $1$.
\end{enumerate}
By using this algorithm one can write the following equations which shows that the algorithm gives a sample of stationary distribution of the MC:

\begin{align*}
\mathbb{P}(Y=y)&= \sum_{T \in \mathcal{T}} P(T) \times \frac{T(\{y\})}{m(T)} \times \frac{m(T)}{M} \times \sum_{n=1}^{\infty} \left( \sum_{T \in \mathcal{T}} P(T) \sum _{y \in T} \frac{T(\{y\})}{m(T)}(1-\frac{m(T)}{M}) \right) ^{n-1}\\
&= \frac{\mathbb{E}(\tau(y))}{M} \times \sum_{n=1}^{\infty} \left( \sum_{T \in \mathcal{T}} P(T) (1-\frac{m(T)}{M}) \right) ^{n-1}\\
&=\frac{\mathbb{E}(\tau(y))}{M}\times \sum_{n=1}^{\infty} \left( 1-\sum_{T \in \mathcal{T}} P(T) \frac{m(T)}{M} \right) ^{n-1}\\
&= \frac{\mathbb{E}(\tau(y))}{M}  \times  \sum_{n=1}^{\infty} \left( 1- \frac{\mathbb{E}(\tau(\mathcal{S}))}{M} \right) ^{n-1}\\
&= \frac{\mathbb{E}(\tau(y))}{M}\times  \frac{M}{\mathbb{E}(\tau(\mathcal{S}))}
= \sigma(y),
\end{align*}
where $\sigma$ is the stationary distribution of the MC, $X$. 

\subsection{ Dynamics on the Space of Random Measures}
\label{Whole}

So far, two dynamics $H^{T}$ and $H^{P}$ have been considered on the Bridge Graph. Here, these dynamics are studied on the general state space, i.e., the space of all integer-valued random measures on $\mathcal{S}$.\\
In the Bridge Graph (or Doeblin Graph), at each time $t$, the Taboo PP and Potential PP are random measures on $\mathcal{S}$. Based on the definitions of $H^{T}$ and $H^{P}$, each of these measures at time $t$ depends only on the measure at time $t-1$ and  $\xi_{t-1}$. So one can consider these dynamics as MCs on the space of $\mathcal{N(S)}$, of all locally finite integer-valued measures on $\mathcal{S}$, and study the properties of these MCs to understand the properties of the dynamics in this more general space. \\
\begin{definition}
\label{9-2}
Consider the Taboo/Potential dynamics constructed by a positive (resp. null) recurrent Markov Chain. The MC, $\Phi ^{T}$/$\Phi ^{P}$, corresponding to this dynamics is called positive (resp. null) recurrent Taboo/Potential Markov Chain on $\mathcal{N(S)}$. 
\end{definition}
The first property that one can consider is the existence of stationary measures of the MCs. It is easy to see that the Taboo PP is a stationary distribution of the positive/null recurrent Taboo MC. Moreover all finite measures are in the domain of attraction of this stationary distribution. Also the Potential PP of a null recurrent MC (in the case where it is a.s. finite at each point), is a stationary distribution of the null recurrent Potential MC. Note that the Potential PP is not a stationary distribution of the positive recurrent Potential MC. Since the Potential PP in this case has an infinite mass a.s. at one point.  So its support does not belongs to $\mathcal{N(S)}$. So the positive/null recurrent Taboo MC and the null recurrent Potential MC have stationary distributions. The question that arises here is about the uniqueness of this stationary distribution, and consequently the irreducibility of these measure-valued MCs. 
The following example shows that there is an invariant measure for the Taboo MC which is not the Taboo PP. 
\begin{example}
\label{9-4}
Consider the Renewal MC in Definition \ref{5-0-1}. Consider the Taboo MC of this MC.
This measure-valued MC is called the \textbf{ Renewal Taboo MC}, denoted by $\{\Phi_{n}^{T,R}\}$. 
In both cases where the Renewal MC is positive recurrent and null recurrent, the Taboo PP is a
stationary distribution of this Renewal Taboo MC. Moreover, any finite measure is in the domain
of attraction of this distribution. \\
First, consider the null recurrent case. Denote the Taboo PP by $\tau _{Ren}$.
Consider the following measure on $\mathbb{N}$:
\begin{equation}
\label{9-7}
\Phi_{0}^{T,R} = \sum _{k\in \mathbb{N}} 1_{k}.
\end{equation}
Select this as the initial state of the MC $\{\Phi_{n}^{T,R}\}$. Then 
\begin{equation}
\label{9-5}
\lim_{n \to \infty} \Phi_{n}^{T,R} = \tau _{Ren} + \sum _{k\in \mathbb{N}} 1_{k}. 
\end{equation}
The same equality holds for the positive recurrent case, when the Renewal Taboo MC
starts with the same initial state $\Phi_{0}^{T,R} $. 
\end{example}
For the null recurrent Potential MC, with the same proof as in Example \ref{9-4},
one can show that considering the null recurrent Potential MC constructed by the null Renewal MC
and starting from the \eqref{9-7} measure, the limit distribution is different from the Potential PP.
So with the same argument, in the Renewal example, the null recurrent Potential MC does not have unique
limiting distribution. So the positive recurrent and null recurrent Taboo MC and the null recurrent
Potential MC cannot be irreducible. \\
\begin{remark} Note that $\mathcal{N}(S)$ is a topological (Polish) space \cite{Kall} that is not countable.
So the concepts and notation of MCs on topological state spaces in \cite{Meyn}
have to be used for considering the properties of these measure-valued MCs.
For example the classical definition of irreducibility has to be replaced by $\psi-irreducibility$.
This definition is with respect to a measure $\psi$ on the state space of the Markov Chain, namely $\mathcal{M(S)}$.  
\end{remark} 
\section{Appendix}
\subsection{Proof of Proposition \ref{T1}}
\label{Proof}
Before going through the proof of this proposition, first Definition \ref{D02} and Lemma \ref{L01} , borrowed from \cite{Kemperman1973}, are discussed. This last lemma gives the main idea of the proof of Proposition \ref{T1}. 
\begin{definition}
\label{D02}
Let $\mu,\nu$, and $\gamma$ be given probability measures on $\mathbb{Z}$. Consider the Markov Chain $\{Y_{n}\}$ with values in $\mathbb{Z}$ such that $Y_{0}=y$ and it has following transition probabilities:
\begin{equation}
\label{5-0}
P(Y_{n+1}=k|Y_{n}=j)=
\begin{cases}
\mu(k-j) & \text{if  $j<0$} \\
\nu(k-j) & \text{if  $j>0$ } \\
\gamma =\alpha \mu(k) + \beta \nu(k)  & \text{if $ j=0$ }, \\
\end{cases}
\end{equation}
where $\alpha,\beta \geq 0 $ and $\alpha + \beta = 1$. This Markov Chain is an ordinary
random walk on $\mathbb{Z}$ with jump distribution $\mu$ in the positive integers,
distribution $\nu$ in the negative integers, and $\gamma$ at $0$. This random walk will
be referred to as the \textbf{oscillating random walk} on $\mathbb{Z}$. \\
The particular case where $\nu (i) = \mu (-i)$ is called the
\textbf{anti-symmetric oscillating random walk }. If moreover $\mu(j)=0$ for $j<0$,
then it is called the \textbf{one-sided anti-symmetric }case. 
\end{definition}
The following lemma from \cite{Kemperman1973} will be used to study the recurrence and
transience property of the oscillating random walk.
\begin{lemma}
\label{L01}
Consider the one-sided antisymmetric oscillating random walk $\{Y_{n}\}$ on $\mathbb{Z}$.
Then a sufficient condition for zero to be recurrent is that 
\begin{equation}
\label{5-2}
\sum_{j=n}^{\infty} \mu (j)=O(n^{-\frac{1}{2}}) \quad \text{as $n \to \infty$.}
\end{equation}
A sufficient condition for zero to be transient is 
\begin{equation}
\label{5-3}
\mu (n) \sim cn^{-1-\epsilon} \quad \text{as $n \to \infty$,}
\end{equation}
where $c$ and $\epsilon$ denote positive constants, $\epsilon < \frac{1}{2}$.
\end{lemma}
\begin{proof} [Proof of Proposition \ref{T1}]
Consider two independent i.i.d. random sequences 
\begin{align*}
\{X_{i}\}_{i \in \mathbb{N}^{*}}, \& \quad \{X_{i}\} \sim \eta  \\
\{X'_{i}\}_{i \in \mathbb{N}^{*}}, \& \quad \{X'_{i}\} \sim \eta   
 \end{align*}
and two random walks on $\mathbb{Z}$ with jumps $\{X_{i}\}$ and  $\{X'_{i}\}$ respectively,
with two arbitrary different starting points $X_{0}$ and $X'_{0}$, namely,
\begin{equation}
\label{4-3-1}
S^{l}=\sum_{i=0}^{l} \{X_{i}\} \quad \quad  S'^{l}=\sum_{i=0}^{l} \{X'_{i}\} . 
\end{equation}
The paths created by these two random walks are two paths in $G^{\eta}$ starting from the
two vertices $X_{0}$ and $X'_{0}$. For checking the connectedness of $G^{\eta}$,
it is needed to check whether these two random walks meet each other in finite time a.s. or not. 
To this end, the MC $\{Z_{n}\}_{n \in \mathbb{N}}$ will be considered.
For $X_{0}<X'_{0}$ define $Z_{0} = X'_{0}-X_{0}$. Moreover fix the $S'^{l}$ at $X'_{0}$
and define $Z_{i}$, the difference between $X'_{0}$ and $X_{i}$ up to the time that $X_{i}$ passes $X'_{0}$,
i.e., 
\begin{equation*}
Z_{i}=X'_{0}-X_{i}, \quad \text{ for $0< i \leq t_{1}$},
\end{equation*}
where $t_{1}$ is the first $t$ such that $X_{t}>X'_{0}$.
Then fix $S^{l}$ at $X_{t_{1}}$ and look at the next steps of $S'^{l}$.
For $i>t_{1}$ define $Z_{i}$, the difference between $X_{t_{1}}$ and $X'_{i}$ up
to the time that $X'_{i}$ passes $X_{t_{1}}$, i.e.,
\begin{equation*}
Z_{i}=X'_{i}-X_{t_{1}},  \quad  \text{for  $t_{1} < i \leq t_{2}$},
\end{equation*}
where $t_{2}$ is the first time where $X'_{t}>X_{t_{1}}$.
After that, fix $X'_{t_{2}}$ and continue this process.
With this definition, $\{Z_{n}\}_{n \in \mathbb{N}}$ is a random walk on $\mathbb{Z}$
which has following transition probability
\begin{equation}
\label{5-1}
P(Z_{n+1}=k|Z_{n}=j)=
\begin{cases}
q_{k-j}& \text{if  $j<0$ and $k>j$ } \\
q_{j-k} & \text{if  $j\geq 0$ and $k<j$ } \\
0  & \text{otherwise }, \\
\end{cases}
\end{equation}
where $\{q_{i}\}$ is the probability defined in \eqref{4-0}. Our question about the meeting
of the two random walks $S^{l}$ and $S'^{l}$ reduces to understanding whether the state $0$
in $\{Z_{n}\}_{n \in \mathbb{N}}$ is recurrent or not. But $\{Z_{n}\}_{n \in \mathbb{N}}$ is
a one sided antisymmetric oscillating random walk where $\mu(j) $  in \eqref{5-0} is equal
to $q_{j}$ and $\beta=1$. So
\begin{equation*}
\sum_{j=n}^{\infty}q_{j}=\sum_{j=n}^{\infty}\frac{c_{1}}{j^{\alpha +1}},
\end{equation*}
which is $O(n^{-\frac{1}{2}})$, as $n \to \infty$, when $\alpha \geq \frac{1}{2}$. 
So using Lemma \ref{L01}, one can conclude that $\{Z_{n}\}$ is recurrent
when $\alpha \geq \frac{1}{2}$, and it is transient when $\alpha < \frac{1}{2}$.
So the two random walks $\{S^{l}\}$ and $\{S'^{l}\}$ will meet each other a.s. when 
$\frac{1}{2} \leq \alpha <1$ and, in this case, $G^{\eta}$ is a Renewal EFT.
Correspondingly, when $\alpha < \frac{1}{2}$, $G^{\xi}$ is a Renewal EFF.
\end{proof}

\section*{Acknowledgements}

This work was supported by the ERC NEMO grant, under the European Union's Horizon 2020 research and innovation programme,
grant agreement number 788851 to INRIA.\\
The authors would like to thank Ali Khezeli and Hermann Thorisson for their comments and suggestions.

\bibliographystyle{imsart-number}
\bibliographystyle{unsrt} 
\bibliography{references}

\begin{thebibliography}{15}

\bibitem{Aldous}
\begin{barticle}[author]
\bauthor{\bsnm{Aldous},~\bfnm{D.}\binits{D.}} \AND
  \bauthor{\bsnm{Lyons.},~\bfnm{R.}\binits{R.}}
(\byear{2007}).
\btitle{Processes on Unimodular Random Networks}.
\bjournal{Electronic Journal of Probability}
\bvolume{12}
\bpages{1454--1508}.
\end{barticle}
\endbibitem

\bibitem{Baccelli2017}
\begin{barticle}[author]
\bauthor{\bsnm{Baccelli},~\bfnm{F.}\binits{F.}},
  \bauthor{\bsnm{Haji-Mirsadeghi},~\bfnm{M.~O.}\binits{M.~O.}} \AND
  \bauthor{\bsnm{Khezeli},~\bfnm{A.}\binits{A.}}
(\byear{2018}).
\btitle{Eternal Family Trees and Dynamics on Unimodular Random Graphs}.
\bjournal{Contemporary Mathematics}
\bpages{85-127}.
\end{barticle}
\endbibitem

\bibitem{Baccelli2018}
\begin{barticle}[author]
\bauthor{\bsnm{Baccelli},~\bfnm{F.}\binits{F.}},
  \bauthor{\bsnm{Haji-Mirsadeghi},~\bfnm{M.~O.}\binits{M.~O.}} \AND
  \bauthor{\bsnm{Murphy},~\bfnm{J.~T.}\binits{J.~T.}}
(\byear{2019}).
\btitle{Doeblin Trees}.
\bjournal{Electronic Journal of Probability}
\bvolume{24}
\bpages{1 -- 36}.
\end{barticle}
\endbibitem

\bibitem{EX}
\begin{barticle}[author]
\bauthor{\bsnm{Bentley},~\bfnm{J.}\binits{J.}} \AND
  \bauthor{\bsnm{Yao},~\bfnm{A.~Chi-Chih}\binits{A.~C.-C.}}
(\byear{1976}).
\btitle{An almost optimal algorithm for unbounded searching}.
\bjournal{Information Processing Letters}
\bvolume{5}
\bpages{82--87}.
\end{barticle}
\endbibitem

\bibitem{Pierre}
\begin{bbook}[author]
\bauthor{\bsnm{Br{\'e}maud},~\bfnm{P.}\binits{P.}}
(\byear{1999}).
\btitle{Markov Chains, Gibbs Fields, Monte Carlo Simulation, and Queues}.
\bpublisher{Springer, New York, NY}.
\end{bbook}
\endbibitem

\bibitem{Monotone}
\begin{barticle}[author]
\bauthor{\bsnm{Daley},~\bfnm{D.}\binits{D.}}
(\byear{December 1, 1968}).
\btitle{Stochastically monotone {M}arkov Chains}.
\bjournal{Zeitschrift f{\"u}r Wahrscheinlichkeitstheorie und Verwandte Gebiete}
\bvolume{10}
\bpages{305-317}.
\end{barticle}
\endbibitem

\bibitem{Paolo}
\begin{barticle}[author]
\bauthor{\bsnm{Fill},~\bfnm{J.~A.}\binits{J.~A.}} \AND
  \bauthor{\bsnm{Machida},~\bfnm{M.}\binits{M.}}
(\byear{2001}).
\btitle{Stochastic Monotonicity and Realizable Monotonicity}.
\bjournal{The Annals of Probability}
\bvolume{29}
\bpages{938--978}.
\end{barticle}
\endbibitem

\bibitem{Tom}
\begin{barticle}[author]
\bauthor{\bsnm{Hutchcroft},~\bfnm{T.}\binits{T.}}
(\byear{2020}).
\btitle{Non-intersection of transient branching random walks}.
\bjournal{Probability Theory and Related Fields}
\bvolume{178}.
\end{barticle}
\endbibitem

\bibitem{Kall}
\begin{bbook}[author]
\bauthor{\bsnm{Kallenberg},~\bfnm{O.}\binits{O.}}
(\byear{2017}).
\btitle{Random Measures, Theory and Applications}
\bvolume{77}.
\bpublisher{Springer}.
\end{bbook}
\endbibitem

\bibitem{Kemperman1973}
\begin{barticle}[author]
\bauthor{\bsnm{Kemperman},~\bfnm{J.~H.~B.}\binits{J.~H.~B.}}
(\byear{1974}).
\btitle{The Oscillating Random Walk}.
\bjournal{Stochastic Processes and their Applications , Elsevier}
\bvolume{2}
\bpages{1--29}.
\end{barticle}
\endbibitem

\bibitem{ALI}
\begin{barticle}[author]
\bauthor{\bsnm{Khezeli},~\bfnm{A.}\binits{A.}}
(\byear{2018}).
\btitle{SHIFT-COUPLING OF RANDOM ROOTED GRAPHS AND NETWORKS}.
\bjournal{Contemporary Mathematics}
\bvolume{719}
\bpages{175-211}.
\end{barticle}
\endbibitem

\bibitem{loynes_1962}
\begin{barticle}[author]
\bauthor{\bsnm{Loynes},~\bfnm{R.~M.}\binits{R.~M.}}
(\byear{1962}).
\btitle{The stability of a queue with non-independent inter-arrival and service
  times}.
\bjournal{Mathematical Proceedings of the Cambridge Philosophical Society}
\bvolume{58}
\bpages{497--520}.
\bdoi{10.1017/S0305004100036781}
\end{barticle}
\endbibitem

\bibitem{Meyn}
\begin{bbook}[author]
\bauthor{\bsnm{Meyn},~\bfnm{S.~P.}\binits{S.~P.}} \AND
  \bauthor{\bsnm{Tweedie},~\bfnm{R.~L.}\binits{R.~L.}}
(\byear{2005}).
\btitle{Markov Chains and Stochastic Stability}.
\bpublisher{Springer}.
\end{bbook}
\endbibitem

\bibitem{Prop}
\begin{barticle}[author]
\bauthor{\bsnm{Propp},~\bfnm{J.}\binits{J.}} \AND
  \bauthor{\bsnm{Wilson},~\bfnm{D.}\binits{D.}}
(\byear{1996}).
\btitle{Exact sampling with coupled {M}arkov chains and applications to
  statistical mechanics}.
\bjournal{Random Structures and Algorithms}
\bvolume{9}
\bpages{223-252}.
\end{barticle}
\endbibitem

\bibitem{Stoyan}
\begin{bbook}[author]
\bauthor{\bsnm{Sung Nok~Chiu},~\bfnm{Dietrich~Stoyan}\binits{D.~S.}}
(\byear{2013}).
\btitle{Stochastic Geometry and Its Applications},
\bedition{third} ed.
\bpublisher{John Wiley, and Sons, New York}.
\end{bbook}
\endbibitem

\end{thebibliography}
\renewcommand{\bibname}{References}

\end{document}